\documentclass[11pt,a4paper]{amsart}

\usepackage[textsize=tiny,textwidth=0.5in]{todonotes}
\usepackage{xcolor}
\usepackage{amsmath}
\usepackage{mathtools}
\usepackage{amsthm}
\usepackage{amscd}
\usepackage{stmaryrd}
\usepackage{amssymb}
\usepackage{enumerate}
\usepackage{color}
\usepackage{amssymb}
\usepackage{mathrsfs}
\usepackage[colorlinks={true},linkcolor={black},citecolor={black}, urlcolor={violet} ]{hyperref}
\usepackage[capitalize,nameinlink,noabbrev]{cleveref}
\usepackage{tikz-cd}

\numberwithin{equation}{section}

\newtheorem{theorem}[equation]{Theorem}
\newtheorem{proposition}[equation]{Proposition}
\newtheorem{lemma}[equation]{Lemma}
\newtheorem{corollary}[equation]{Corollary}

\newtheorem{sublemma}[equation]{Sublemma}

\theoremstyle{definition}
\newtheorem{definition}[equation]{Definition}
\newtheorem{remark}[equation]{Remark}
\newtheorem*{remark*}{Remark}
\newtheorem{example}[equation]{Example}

\newcommand{\p}{\mathfrak{p}}
\newcommand{\calO}{\mathcal{O}}
\DeclareMathOperator{\Spec}{Spec}
\DeclareMathOperator{\Sp}{Sp}
\DeclareMathOperator{\Hom}{Hom}
\DeclareMathOperator{\rad}{rad}

\newcommand{\bD}{\widetilde{D}}
\newcommand{\wA}[1]{{S(#1)}}

\title[Scheme Theory for commutative semirings]{Scheme Theory for commutative semirings}

\author[R. Gualdi]{Roberto Gualdi}
\address{\rm Roberto Gualdi, Departament de Matem\`atiques, Universitat Polit\`ecnica de Catalunya, 08034 Barcelona, Spain}
\email{roberto.gualdi@upc.edu}

\author[A. Kuhrs]{Arne Kuhrs}
\address{\rm Arne Kuhrs, Institut f\"ur Mathematik, Universit\"at Paderborn, 33098 Paderborn, Germany}
\email{kuhrs@math.uni-paderborn.de}

\author[M. Mayo]{Mayo Mayo Garcia}
\address{\rm Mayo Mayo Garcia, University of Warwick, Coventry CV4 7AL, United Kingdom }
\email{mayo.mayo-garcia@warwick.ac.uk}

\author[X. Xarles]{Xavier Xarles}
\address{\rm Xavier Xarles, Departament de Matem\`atiques\\Universitat Aut\`onoma de
		Barcelona\\08193 Bellaterra, Barcelona, Catalonia}
\email{xavier.xarles@uab.cat}

\begin{document}

\begin{abstract}
In this survey, we describe two different approaches to constructing affine schemes for commutative semirings: one based on prime ideals, and another based on prime kernels (also called subtractive ideals).

We then explain how these two approaches are related through the theory of universal valuations.
\end{abstract}

\maketitle

\setcounter{tocdepth}{1}
\tableofcontents

\section*{Introduction}

The main idea in these pages is to discuss possible definitions for the affine scheme associated to a commutative semiring, in analogy with the usual functorial construction from algebraic geometry, see for example ~\cite{Hartshorne} or \cite{stacks-project}.
In practice, one would like to associate to a semiring $A$ a semiringed space~$(X,\mathcal{O}_X)$, which should be considered as its geometric realization.

However, unlike the case of rings, the construction presents several difficulties.
In fact, already at the level of the underlying topological space~$X$, one is faced with a choice: it could be defined as the set of all prime ideals of $A$ or instead one could restrict to the collection of its kernels (see \cref{def: substractive ideal}).
The resulting topological spaces are denoted by $\Spec A$ and $\Sp A$, respectively, and they can be substantially different; see, for instance, \cref{ex: Sp and Spec of N,ex: Sp and Spec of B[x]} and the related \cref{rem: dimension of semiring spectrum}.

Having made a choice for~$X$, the problem of defining a sheaf of semirings $\mathcal{O}_X$ on it is also nontrivial.
In fact, while the construction works as expected for~$X=\Spec A$ and its local sections are easily computable on a basis for the topology, in the case of $X=\Sp A$ some new obstructions appear. 

Relevantly, the approach using all prime ideals yields a category of affine semischemes that is equivalent to the opposite category of commutative semirings (\cref{prop: category Spec is equivalent}).
On the other hand, we do not know whether the category of affine schemes obtained only using kernels is equivalent to some meaningful subcategory of commutative semirings.
We suggest this could be the category of the so-called \emph{global semirings} studied in \cref{subs: Global}.


\vspace{\baselineskip}

The notes are organized as follows.
In \cref{sec: basic definitions} we recall several basic definitions and properties about semirings and their ideals.
We study the two possibilities for the construction of the topological space $X$ associated to a semiring $A$ in \cref{sec: spectrum constructions}.

Then, we move to the definition of the structure sheaf.
In \cref{sec: structure sheaf for Spec} we give three different definitions for $\mathcal{O}_X$ on~$X=\Spec A$ and show that they are equivalent.
In \cref{sec: structure sheaf for Sp} we study the analogous problem for~$X=\Sp A$.

Finally, in \cref{sec:universalGval}, we study $G$-valuations, which highly generalize usual valuations to the context of semirings and give an unexpected link between the two scheme theories.

\vspace{\baselineskip}

\textbf{Acknowledgments.}
These notes are a transcription and expansion of a minicourse delivered by the fourth author in occasion of the conference \textit{Geometry over Semirings} held at the Universitat Autònoma de Barcelona in July 2025.
We thank the organizers of the conference, Marc Masdeu and Joaquim Roé, for creating a friendly and inspiring atmosphere for work and exchange. 

The last author also thanks the participants for numerous comments and suggestions, and specially Oliver Lorscheid for some references, and Stefano Mereta for pointing out the paper \cite{BarilBoudreau_Garay} and for answering some questions.

Roberto Gualdi was partially supported by the MICINN research project PID2023-147642NB-I00 and the AGAUR research project 2021-SGR-00603.

Arne Kuhrs has received funding from the Deutsche Forschungsgemeinschaft (DFG, German Research Foundation) TRR 358 \emph{Integral Structures in Geometry and Representation Theory}, project number 491392403, as well as from the Deutsche Forschungsgemeinschaft (DFG, German Research Foundation) Sachbeihilfe \emph{From Riemann surfaces to tropical curves (and back again)}, project number 456557832 and the DFG Sachbeihilfe \emph{Rethinking tropical linear algebra: Buildings, bimatroids, and applications}, project number 539867663, within the
SPP 2458 \emph{Combinatorial Synergies}.

Mayo Mayo García is supported by the Warwick Mathematics Institute Centre for Doctoral Training, and gratefully acknowledges funding from the University of Warwick.

Xavier Xarles is partially supported by the Spanish State Research Agency MICIU/ AEI / 10.13039/501100011033 (Projects PID2020-116542GB-I00 and PID2024-159095NB-I00).

\section{Semirings and their ideals}\label{sec: basic definitions}

In this section, we recall the basics of the objects playing a role in the article.
In particular, we revise the definition of a semiring and its ideals.
We then focus on a specific class of ideals in a semiring, namely kernels, and discuss their features.
Finally, we compare the properties of invertibility and semi-invertibility in a semiring and introduce the notion of hardness. More details on the theory of semirings and semimodules can be found in \cite{Golan1999}.

\subsection{Semirings and localizations}

A semiring is nothing else than a ring without the requirement for additive inverses, the precise definition being as follows.

\begin{definition}
A \emph{(commutative) semiring} is a set $A$ with two binary operations $+$ and $\cdot$ such that both $(A,+)$ and $(A,\cdot)$ are commutative monoids with identity ($0_A$ and $1_A$ respectively), satisfying the distributivity of the product with respect to addition, and the property $a\cdot 0_A=0_A$ for all~$a\in A$.

A semiring $A$ is called \emph{idempotent} if $a+a=a$ for all $a\in A$.
\end{definition}

Notice that, in virtue of the distributive property, a semiring is idempotent if and only if~$1_A+1_A=1_A$.
Moreover, if $A$ is an idempotent semiring, it naturally carries a partial order defined by $a\preceq b$ whenever~$a+b=b$.

\begin{example}\label{ex: semirings}
Many algebraic structures from different parts of mathematics are actually semirings.
For instance:
\begin{enumerate}
    \item Every commutative ring with unity is a semiring. It is idempotent if and only if it is the trivial ring.
    \item The set of natural numbers $\mathbb{N}=\{0,1,\ldots\}$ with the usual addition and multiplication is a semiring; it is not idempotent.
    \item The logic triple $\mathbb{B}=(\{0,1\},\lor,\land)$ or more generally any \emph{boolean algebra}, see for instance \cite[Section~4.26]{Davey_Priestley} for the precise definition, is an idempotent semiring.
    \item The \emph{tropical semiring}
    \[
    \mathbb{T}=(\mathbb{R}\cup\{+\infty\},\min,+).
    \]
    is an idempotent semiring, which is the main motivation for these notes and at the heart of current trends in algebraic geometry. We also refer the reader to \cite{Lorscheid_this_volume} and \cite{Giansiracusa_this_volume} for different approaches to tropical geometry.
    \item If $A$ is a semiring, the set $A[x_1,\ldots,x_n]$ of polynomials in $n$ variables with coefficients in $A$ is naturally equipped with a semiring structure. The polynomial semiring $A[x_1,\ldots,x_n]$ is idempotent if and only if $A$ is.
\end{enumerate}

\end{example}

Basic notions from ring theory are naturally translated to the setting of semirings. For example, the theory of modules is essentially the same (except for some results concerning the quotient by submodules; see for example \cref{sec: subtractive} for the case of submodules of the semiring). 

Given two semirings $A$ and~$B$, a \emph{homomorphism} $f\colon A\to B$ is a map preserving addition and multiplication, and such that $f(0_A)=0_B$ and~$f(1_A)=1_B$. An \emph{isomorphism} of semirings is a semiring homomorphism that is both injective and surjective. However, this is not equivalent to the morphism being surjective and having a kernel equal to 0 (consider, for example, $f\colon(\{0, 0.5,1\}, \max, \min)\rightarrow (\{0,1\}, \max, \min)$ with $f(0)=0$ and $f(1)=f(0.5)=1$).

Moreover, an element $a$ of a semiring $A$ is called \emph{invertible} if it admits a multiplicative inverse.
In case it exists, this element is unique and is denoted by~$a^{-1}$.
A semiring in which all nonzero elements are invertible is called a \emph{semifield}. Note that the subset of invertible elements $A^{\times}$ is a multiplicative subgroup of~$A$. 

\begin{example}
The boolean domain $\mathbb{B}$ and, more interestingly, the tropical semiring $\mathbb{T}$ in \cref{ex: semirings} are idempotent semifields.
\end{example}

There is a well-defined process of \emph{localization} for semirings.
If $A$ is a semiring and $S\subseteq A$ is a multiplicative submonoid, then we set
\[
S^{-1}A=\Big\{\frac{a}{s}\Bigm| a\in A,s\in S\Big\},
\]
where two fractions $a_1/s_1$ and $a_2/s_2$ are identified if and only if there exists $u\in S$ such that~$a_1s_2u=a_2s_1u$.
Analogously to the case of rings, $S^{-1}A$ inherits from $A$ the structure of a semiring, and it comes with a natural homomorphism $\varphi_S\colon A\to S^{-1}A$ that fulfills the usual universal property. 

Finally, a multiplicative submonoid $S\subseteq A$ is said to be \emph{saturated} if it verifies that for any $a,b$ in $A$ with~$ab\in S$, then $a \in S$ and~$b\in S$. 
For an arbitrary multiplicative submonoid~$S\subseteq A$, we call
\[
S^{\mathrm{sat}}=\{b\in A \mid \exists c\in A \text{ such that } b c\in S\}
\]
the \emph{saturation} of~$S$.
Its name is justified by the following elementary lemma.

\begin{lemma}\label{lem: Saturation}
Let $S\subseteq A$ be a multiplicative submonoid.
Then $S^{\mathrm{sat}}$ is the smallest saturated multiplicative submonoid of $A$ that contains~$S$.
Moreover,
\[S^{\mathrm{sat}} = \{b\in A \mid \varphi_S(b) \in (S^{-1}A)^{\times} \}\] 
and the natural morphism $(S^{\mathrm{sat}})^{-1}A \to S^{-1}A$ obtained from the universal property of the localization at $S^{\mathrm{sat}}$ is an isomorphism.
\end{lemma}  

\begin{proof}
We first show that $S^{\mathrm{sat}}$ is a multiplicative submonoid of~$A$.
Since $1_A\in S$ we have $1_A\in S^{\mathrm{sat}}$.
Moreover, if $b_1,b_2\in S^{\mathrm{sat}}$, choose $c_1,c_2\in A$ such that $b_1c_1,b_2c_2\in S$. Then $(b_1b_2)(c_1c_2)= (b_1c_1)(b_2c_2)\in S$. 

Further, $S^{\mathrm{sat}}$ is saturated. Indeed, suppose $ab\in S^{\mathrm{sat}}$.  Then there exists $c\in A$ such that $(ab)c\in S$. 
If we set $c'=bc$, then $a c'\in S$, hence $a\in S^{\mathrm{sat}}$. Similarly, $b\in S^{\mathrm{sat}}$. 

To see minimality, let $T\subseteq A$ be any saturated multiplicative submonoid containing~$S$.
Let $b\in S^{\mathrm{sat}}$, then there exists $c \in A$ with $bc\in S\subseteq T$. Since $T$ is saturated it follows that~$b\in T$.
Hence~$S^{\mathrm{sat}}\subseteq T$. 

Now, if $b\in S^{\mathrm{sat}}$, there exists $c\in A$ such that $bc = s$ for some $s \in S$, so
\[
\varphi_S(b)\varphi_S(c)=\frac{bc}{1_A}=\frac{s}{1_A}
\]
which is invertible in $S^{-1}A$. Hence $\varphi_S(b)\in (S^{-1}A)^{\times}$. 
Conversely, suppose $b \in A$ with $\varphi_S(b)\in (S^{-1}A)^{\times}$. Then there exists $a/s\in S^{-1}A$ such that
\[
\frac{1_A}{1_A}=\varphi_S(b)\frac{a}{s}=\frac{ba}{s}.
\]
By definition, there exists a $u\in S$ with $u(ba)=us$. 
Since $b(ua) = us\in S$ we see that~$b\in S^{\mathrm{sat}}$. 

Finally, the natural morphism $(S^{\mathrm{sat}})^{-1}A \to S^{-1}A$ given by the universal property of localization is the one mapping each element $a/b$ to
\[\varphi_S(a)\varphi_S(b)^{-1}=\frac{ac}{bc},\]
where $c$ is any element in $A$ such that~$bc\in S$.
It is immediately verified that $S^{-1}A\to(S^{\mathrm{sat}})^{-1}A$ defined by $a/s\mapsto a/s$ is an inverse morphism.
\end{proof}

\subsection{Ideals and prime ideals}

Let $A$ denote a semiring.
As in the case of rings, it is clear how to define a semimodule over~$A$.

\begin{definition}
Let $A$ be a semiring.  
An \emph{$A$‑semimodule} is a commutative monoid $(M,+,0_M)$ together with a scalar multiplication
\[
A\times M \longrightarrow M, \qquad (a,m)\longmapsto a\cdot m,
\]
such that for all $a,b\in A$ and $m,n\in M$,
\begin{align*}
  & a\cdot (m+n)=a\cdot m+a\cdot n,\\
  & (a+b)\cdot m=a\cdot m+b\cdot m,\\
  & (a b)\cdot m=a\cdot(b\cdot m),\\
  & 1_A\cdot m=m,\qquad 0_A\cdot m=0_M,\qquad a\cdot0_M=0_M.
\end{align*}
A \emph{morphism of $A$‑semimodules} $f\colon M \to N$ is a  morphism of the underlying additive monoids
$f\colon M\to N$ satisfying $f(a\cdot m)=a\cdot f(m)$ for all $a\in A$, $m\in M$.
\end{definition}

\begin{definition}\label{def:subsemimodule}
An \emph{$A$-subsemimodule} of an $A$‑semimodule $M$ is a subset $N\subseteq M$ that contains $0_M$ and is closed under addition and scalar multiplication by~$A$.
\end{definition}

A particular instance of this will be at the center of our treatment.

\begin{definition}
An \emph{ideal} $I$ of $A$ is an $A$-subsemimodule of~$A$.
More explicitly, $I$ is a subset of $A$ containing $0_A$ and satisfying $I + I \subseteq I$ and~$A\cdot I\subseteq I$.

An ideal $I$ of $A$ is said to be \emph{prime} if $I \neq A$ and whenever $a,b\in A$ are such that~$a b\in I$, then $a\in I$ or~$b\in I$.
Equivalently, $I$ is prime if and only if $A\setminus I$ is a nonempty (saturated) multiplicative submonoid.  
\end{definition}

Obviously, these notions agree with the usual ones when the semiring $A$ is actually a ring.
On the contrary, in the general case we can have some more surprising examples for algebraic geometers.

\begin{example}\label{ex: a prime ideal of N}
    Let $\mathbb{N}$ be the semiring of natural numbers and consider
    \[
    I=\mathbb{N}\setminus\{1\}=\{0,2,3,\ldots\}.
    \]
    The set $I$ is closed under addition, it contains~$0$, and any natural multiple of an element of $I$ still belongs to~$I$.
    Moreover, if a natural number is different from $1$ then it has at least one nontrivial factor.
    It follows from all these observations that $I$ is a prime ideal of the semiring~$\mathbb{N}$. 
\end{example}

Given a finite number of elements $a_1,\ldots,a_r\in A$, the set
\[
\langle a_1,\ldots,a_r\rangle=a_1A+\ldots+a_rA=\{a_1b_1+\ldots+a_rb_r \mid b_1,\ldots,b_r\in A\}
\]
is an ideal in~$A$, called the \emph{ideal generated by~$a_1,\ldots,a_r$}.
It is the smallest ideal of $A$ containing~$a_1,\ldots,a_r$.

\begin{remark}\label{rem: ideal generated by invertibles}
    An element $a\in A$ is invertible if and only if~$aA=A$.
    Indeed, if $a$ is invertible, then for all $b\in A$ we have $b=a\cdot a^{-1}\cdot b$ and so $b$ belongs to~$aA$.
    Conversely, if $aA=A$ then $1_A\in aA$ and so there is a multiplicative inverse of $a$ in~$A$.
\end{remark}

\begin{remark}\label{rem: ideals of semifields}
    If $F$ is a semifield, then the unique ideals of $F$ are $\{0_F\}$ and~$F$.
    Indeed, let $I$ be a non-zero ideal of~$F$, and let $a\neq 0_F$ be an element of~$I$.
    It is invertible, and hence by \cref{rem: ideal generated by invertibles} the minimal ideal of $F$ containing it is $F$ itself, forcing~$I=F$.
    Notice also that in this case $\{0_F\}$ is the only prime ideal of~$F$.
\end{remark}

It is customary to denote prime ideals of a semiring $A$ with gothic letters.
Notice that for a prime ideal $\p$ of~$A$, since $A\setminus \p$ is a saturated multiplicative submonoid, we can localize at~$A\setminus \p$.
We denote the corresponding localization by~$A_{\p}=(A\setminus \p)^{-1}A$.  

The following statement, whose proof is immediate, reflects the usual stability of ideals and primality under taking preimages.

\begin{lemma}\label{lem: preimage of ideals}
    Let $f\colon A \to B$ be a morphism of semirings, and let $I$ be an ideal of~$B$.
    Then:
    \begin{enumerate}
        \item $f^{-1}(I)$ is an ideal of~$A$;
        \item if $I$ is prime, then $f^{-1}(I)$ is prime.
    \end{enumerate}
\end{lemma}

Finally, we will need the semiring version of Krull's theorem, which equates the radical of an ideal with the intersection of all prime ideals containing it. 

\begin{theorem}\cite[Proposition 7.28]{Golan1999}\label{thm: radical of ideals}
    Let $I$ be an ideal in a semiring~$A$.
    Then the radical of $I$ defined as 
    \[\rad(I)=\{a\in A \mid \exists n\ge 0 \text{ such that } a^n\in I\}\]
    is an ideal and\[\rad(I) =\bigcap_{\substack{I \subseteq \p \\\p\ \mathrm{prime\ ideal}}} \p.\]
\end{theorem}

\subsection{Subtractive ideals and kernels} \label{sec: subtractive}

Let us fix the choice of a semiring $A$ and spend a few words on the relation between ideals of $A$ and morphisms with source~$A$.

On the one hand, if $f\colon A\to B$ is a morphism of semirings, the subset $\ker(f)=f^{-1}(0_B)$ is an ideal in~$A$ because of \cref{lem: preimage of ideals}.
We call it the \emph{kernel} of~$f$.

On the other hand, any ideal $I$ of a semiring $A$ gives rise to a morphism of semirings.
To do so, we define a congruence relation $\sim_I$ on $A$ by declaring, for all~$a,b\in A$,
\[
a \sim_I b \quad\iff\quad \exists\ i,j \in I \text{ such that }a+i = b+ j.
\]
Then $A/{\sim_I}$ inherits from $A$ a semiring structure, and the \emph{quotient map} $\pi_I\colon A \to A/{\sim_I}$ is a semiring homomorphism. 

\begin{remark}\label{rem: relation between kernel and quotient}
In the above notation, it always holds that~$I\subseteq\ker(\pi_I)$, as $i\sim_I0_A$ for all~$i\in I$.
However, it can happen that the containment is strict.
For example, consider the ideal $I=\mathbb{N}\setminus\{1\}$ of the semiring~$\mathbb{N}$, as in \cref{ex: a prime ideal of N}.
It satisfies
\[
\ker(\pi_I)
=
\{n\in\mathbb{N}\mid n\sim_I0\}
=
\{n\in\mathbb{N}\mid n+m=\ell\text{ for some }m,\ell\in\mathbb{N}\setminus\{1\}\}
\]
and so in this case~$I\subsetneq\ker(\pi_I)=\mathbb{N}$.
\end{remark}

To avoid the unpleasant situation of the example mentioned in \cref{rem: relation between kernel and quotient} we introduce the following special class of semiring ideals.

\begin{definition}\label{def: substractive ideal}
An ideal $I$ in a semiring $A$ is called a \emph{subtractive ideal} (or a \emph{$k$-ideal} or also a \emph{kernel of $A$}) if for every $a\in A$ such that there exist $b,c\in I$ with $a+b=c$ then $a\in I$. 
\end{definition}

The different terminology in \cref{def: substractive ideal} is justified by the following equivalence. 

\begin{proposition}\cite[Proposition 10.11]{Golan1999}\label{prop: k-ideals}
Let $I$ be a subset of a semiring~$A$.
Then, the following are equivalent:
\begin{enumerate}[(i)]
\item $I$ is a subtractive ideal of $A$,
\item $I$ is the kernel of the quotient map~$\pi_I \colon A\to A/{\sim_I}$,
\item there exists a morphism of semirings $f\colon A\to B$ such that~$I=\ker(f)$.
\end{enumerate}
\end{proposition}
\begin{proof}
Let us start by supposing that (i) holds.
The kernel of $\pi_I\colon A \to A/{\sim_I}$ contains~$I$, as noticed in \cref{rem: relation between kernel and quotient}.
Conversely, let~$a \in \ker(\pi_I)$.
Then $a\sim_I0_A$ and so there are $i,j \in I$ such that~$a + i = j$.
Since $I$ is subtractive, it follows that~$a \in I$.

The fact that (ii) implies (iii) is obvious.
Hence, assume now that (iii) holds and let $I = f^{-1}(0_B)$ for a morphism of semirings~$f\colon A\to B$.
Then $I$ is an ideal of $A$ by \cref{lem: preimage of ideals}.
Moreover, let $a \in A$ and $b,c \in I$ such that $a + b = c$.
Applying $f$ to both sides of the equality we conclude that~$f(a)=0_B$, and so~$a\in I$.
\end{proof}

Apart from \cref{def: substractive ideal}, the previous result also justifies the following terminology.

\begin{definition}\label{def: subtractive closure}
    Let $I$ be an ideal in a semiring $A$ and $\pi_I\colon A\to A/{\sim_I}$ the corresponding quotient map.
    Then, the ideal $\overline{I}=\ker(\pi_I)$ is called the \emph{subtractive closure of~$I$} (or also, the \emph{kernel of~$I$}).
\end{definition}

Indeed, for any ideal $I$ of~$A$, the set $\overline{I}$ is the smallest subtractive ideal of $A$ that contains~$I$.
The fact that it is a subtractive ideal follows from \cref{prop: k-ideals}.
To show minimality, assume that $I$ is contained in a subtractive ideal $J$ of $A$ and let~$a\in\ker(\pi_I)$.
Then, by definition $a\sim_I 0_A$ and so there exist $b,c\in I\subseteq J$ with $a+b=c$.
Since $J$ is subtractive it must happen that~$a\in J$.
So, any subtractive ideal $J$ containing $I$ has to contain~$\overline{I}=\ker(\pi_I)$.

\begin{remark}\label{rem: subtractive ideals in rings}
When $A$ is actually a ring, it is obvious from \cref{def: substractive ideal} and the existence of additive inverses that all ideals of $A$ are subtractive.
\end{remark}

\begin{remark}\label{rmk: Kernels for idempotent semirings}  
When the semiring $A$ is idempotent, \cref{def: substractive ideal} can be rephrased in the following slightly simpler way: an ideal $I$ of $A$ is subtractive if and only if for every $a\in A$ such that there exists $b\in I$ with $a+b=b$ then~$a\in I$.
It is obvious that such a condition for $a$ implies the one in \cref{def: substractive ideal}.
Conversely, if there exist $b,c\in I$ with $a+b=c$ we can add $b+c$ to both sides of the equality to obtain
\[
a+(b+c)=b+c,
\]
thanks to the fact that $A$ is idempotent.

Recalling that an idempotent semiring $A$ is a partially ordered set with respect to the order defined by $a\preceq b$ whenever $a+b=b$, we obtain that an ideal $I$ of $A$ is subtractive if and only if it is down-closed: if $a\preceq b$ and~$b\in I$, then~$a\in I$.  
\end{remark}

The following property of subtractive ideals will be used later.

\begin{lemma}\label{lem: preimage of subtractive ideal}
Let $f\colon A \to B$ be a morphism of semirings. If $I \subseteq B$ is a subtractive ideal of~$B$, then $f^{-1}(I)$ is a subtractive ideal of~$A$. 
\end{lemma}
\begin{proof}
    Let $a\in A$ and suppose there are $b,c\in f^{-1}(I)$ with~$a+b=c$.
    This implies that~$f(a)+f(b)=f(c)$, and so that~$f(a)\in I$, by definition.
    
    Alternatively, the claim follows from \cref{prop: k-ideals} and the observation that $f^{-1}(I)=\ker(\pi_I\circ f)$, where $\pi_I \colon B\to B/{\sim_I}$ is the quotient map. 
\end{proof}

\subsection{Hard semirings}

We conclude our algebraic discussion by considering a special class of semirings, for which the notion of invertibility is equivalent to a much weaker one.

To do so, recall from \cref{rem: ideal generated by invertibles} that an element $a$ of a semiring $A$ is invertible if and only if~$aA=A$.
This property, together with \cref{def: subtractive closure}, suggests the following relaxed notion.

\begin{definition}
    Let $A$ be a semiring.
    An element $a\in A$ is said to be \emph{semi-invertible} if~$\overline{aA}=A$.
\end{definition}

Since the subtractive closure of $aA$ is an ideal of~$A$, saying that $a$ is semi-invertible is equivalent to asserting that~$1_A\in\overline{aA}$, or more explicitly that there exist $b,c\in A$ with~$1_A+ab=ac$.

\begin{remark}\label{rem: Semi-invertible for idempotent semirings} 
    Using the partial order on an idempotent semiring $A$, we can describe its semi-invertible elements as those $a\in A$ such that there exists $b\in A$ with~$1_A\preceq ab$.

    This is because if for $a,b,c\in A$ we have~$1_A+ac=ab$, then $1_A+ab=1_A+(1_A+ac)=1_A+ac=ab$, hence $1_A\preceq ab$. 
\end{remark}

\begin{remark}\label{rem: subtractive ideals with semi-invertble elements are full semiring}
 If a subtractive ideal $I$ of $A$ contains a semi-invertible element, then~$I=A$.
 Indeed, if $a\in I$ is semi-invertible, then $A=\overline{aA}\subseteq\overline{I}=I$.
\end{remark}

It is obvious from the definition that any invertible element of $A$ is also semi-invertible, and that when $A$ is a ring the two notions agree.
In a general semiring, however, there might be many more semi-invertible elements than invertible ones.

\begin{example}\label{ex: non-hard semiring}
    Consider the semiring $A=\mathbb{B}[x]$ of univariate polynomials with coefficients in the boolean semiring $\mathbb{B}$ from \cref{ex: semirings}.
    Any polynomial in this semiring is a finite sum of monomials with coefficient equal to~$1$, and the multiplication between two such polynomials can be expressed combinatorially as follows: if $S_1,S_2$ are two finite subsets of~$\mathbb{N}$, then
    \begin{equation}\label{eq: multiplication in boolean polynomial semiring}
    \bigg(\sum_{s\in S_1}x^s\bigg)\cdot\bigg(\sum_{s\in S_2}x^s\bigg)
    =
    \sum_{s\in S_1+S_2}x^s.            
    \end{equation}
    It follows that the only invertible element of $A$ is the constant polynomial~$1$.
    
    Instead, an element of $A$ is semi-invertible if and only if it is of the form $1 + f$ with~$f\in A$.
    Indeed, $A$ is idempotent and then
    \[
    1+(1+f)=1+f,
    \]
    showing that any polynomial of the form $1+f$ is semi-invertible.
    Conversely, if $a\in A$ is semi-invertible there should exist $b,c\in A$ with~$1+ab=ac$.
    Since the left-hand side of this equality is a polynomial with a nonzero free term, the same must happen for the right-hand side.
    In view of \eqref{eq: multiplication in boolean polynomial semiring} both $a$ and $c$ must have a nonzero free term.
\end{example}

Generalizing \cref{ex: non-hard semiring}, we have the following description of semi-invertible polynomials over idempotent semifields. 

\begin{lemma}\label{lem: semi-invertibles in polynomial ring}
    Let $F$ be an idempotent semifield.
    Then, an element $f\in F[x_1,\ldots,x_n]$ is semi-invertible if and only if~$f(0_F) \neq 0_F$.
\end{lemma}

\begin{proof}
    First, suppose that we have $f\in F[x_1,\ldots,x_n]$ with $\lambda=f(0_F) \neq 0_F$. Then $\lambda^{-1}f = 1_F + g$ for $g \in F[x_1,\ldots,x_n]$ with $g(0_F)=0_F$, so
    \[1_F+\lambda^{-1}f=1_F+(1_F+g)=1_F+g=\lambda^{-1}f.\]
    In particular, $f$ is semi-invertible. 

    Conversely, if~$f(0_F)=0_F$, then for all $g\in F[x_1,\ldots,x_n]$ we have $(fg)(0_F)=0_F$, while~$(1_F+fg)(0_F)=1_F$.
    Therefore $1_F+fg \ne fg$,  and hence $f$ is not semi-invertible by \cref{rem: Semi-invertible for idempotent semirings}. 
\end{proof}

We reserve a special name for the semirings in which there is no difference between the notion of invertibility and of semi-invertibility.

\begin{definition}
    A semiring $A$ is called \emph{hard} if the set of its invertible elements coincides with the set of semi-invertible ones. 
\end{definition}

\begin{example}\label{ex: hard and non-hard semirings}
    Here are some instances of hard and non-hard semirings.
    \begin{enumerate}
        \item Any ring is hard as a semiring.
        This is a direct consequence of \cref{rem: subtractive ideals in rings}.
        \item The semiring of natural numbers $\mathbb{N}$ is hard, since its only semi-invertible element is~$1$.
        To see this, let $n\in\mathbb{N}$ be semi-invertible.
        Then there exist two natural numbers $m,\ell$ for which
        \[
        1+nm=n\ell.
        \]
        The left-hand side of this equality is a natural number congruent to $1$ modulo~$n$, while the right-hand side is a multiple of~$n$.
        This is absurd unless~$n=1$.
        \item Any semifield is a hard semiring. Indeed a semi-invertible element in a nontrivial semiring is necessarily distinct from zero and so by definition of semifield it is invertible. In particular, the boolean domain $\mathbb{B}$ and the tropical semiring $\mathbb{T}$ are hard.
        \item The semiring $\mathbb{B}[x]$ is not hard because of \cref{ex: non-hard semiring}.
    \end{enumerate} 
\end{example}

The importance of hard semirings in this article is related to the following result, which characterizes ideals whose subtractive closure equals the full semiring.

\begin{proposition}\label{prop: ideals whose kernel is total}
    Let $A$ be a hard and idempotent semiring.
    Then, if $I$ is an ideal of $A$ with~$\overline{I}=A$, we have~$I=A$.
\end{proposition}
\begin{proof}
    By assumption on~$I$, we have that~$1_A\in \overline{I}$.
    Otherwise said, there exist $i,j\in I$ with~$1_A+i=j$.
    Since $A$ is idempotent, we can write
    \[
    j=1_A+i=1_A+1_A+i=1_A+j.
    \]
    In particular, $j$ is semi-invertible and then, by hardness of~$A$, it is invertible.
    It follows that $I$ contains an invertible element of~$A$, and then~$I=A$.
\end{proof}

\begin{example}
    The implication in \cref{prop: ideals whose kernel is total} is not true for general hard semirings.
    For example, consider the semiring~$\mathbb{N}$, which is hard by \cref{ex: hard and non-hard semirings}.
    Its ideal $I=\mathbb{N} \setminus\{1\}$ from \cref{ex: a prime ideal of N} is proper, but it verifies that~$\overline{I}=\mathbb{N}$, as $1\in \overline{I}$ since $1+2=3$ and~$2,3\in I$. 
\end{example}

Given an arbitrary semiring~$A$, there is a canonical way of constructing a hard semiring associated to~$A$.
To do so, denote by $\wA A$ the set of semi-invertible elements of~$A$.

\begin{lemma}\label{lem: semiinvertibles are saturated multiplicative monoid}
    The subset $\wA A$ of semi-invertible elements of~$A$ is a saturated multiplicative submonoid. 
\end{lemma}
\begin{proof}
    It is clear that $1_A\in \wA A$. Moreover, if $a$ and $b$ are semi-invertible, there exist $c,d,x,y\in A$ such that $1_A+ac=ad$ and~$1_A+bx=by$.
    Then, $1_A+ab(cy+dx)=ab(dy+cx)$ and therefore $ab$ is also semi-invertible. 

    Conversely, if $ab$ is semi-invertible, so $1_A+abc=abd$ for some~$c,d\in A$.
    Then clearly $a$ and $b$ are semi-invertible too, as~$1_A+a(bc)=a(bd)$, and analogously for~$b$.  
    \end{proof}

We also refer the reader to \cref{lem: submonoid of open set for kernels} for a generalization of \cref{lem: semiinvertibles are saturated multiplicative monoid} (with an independent proof).
In view of the latter, it makes sense to consider the localization
\begin{equation}\label{eq: definition of hardening}
    A^{\diamondsuit}={(\wA A)}^{-1}A,
\end{equation}
which is called the \emph{hardening} of~$A$.
The terminology is justified by the following result.

\begin{proposition}\label{prop: hardening is hard}
    The semiring $A^{\diamondsuit}$ is hard.
    Moreover, the localization homomorphism $\chi_A\colon A\to A^{\diamondsuit}$ verifies the following universal property: for any homomorphism $f\colon A\to B$ with $B$ hard, there exists a unique homomorphism $f^{\diamondsuit}\colon A^{\diamondsuit}\to B$ such that $f=f^{\diamondsuit}\circ \chi_A$.
    
    In particular, for any morphism $f\colon A\to B$ between two arbitrary semirings, there exists a unique homomorphism $f^{\diamondsuit}\colon A^{\diamondsuit}\to B^{\diamondsuit}$ such that $ f^{\diamondsuit}\circ \chi_A=\chi_B\circ f$.
\end{proposition}
\begin{proof}
    Recall that for a multiplicative submonoid~$S\subseteq A$, an element $a/s\in S^{-1}A$ is invertible if and only if there are $b\in A$ and $t,u\in S$ such that~$abu=stu$. If $S$ is saturated, this is equivalent to saying that~$a\in S$, see also \cref{lem: Saturation}. 

    So, thanks to \cref{lem: semiinvertibles are saturated multiplicative monoid}, to prove that $A^{\diamondsuit}$ is hard it is enough to show that any semi-invertible element $a/s$ of~$A^\diamondsuit=(\wA A)^{-1}A$ satisfies~$a\in \wA A$.
    Since $a/s$ is semi-invertible, there exist $b,c\in A$ and $u,v\in\wA A$, such that 
    \[ \frac{1_A}{1_A} + \frac as \frac bu =\frac as \frac cv.\]
    Therefore, there exists $w\in \wA A$ such that 
    \[s^2uvw + a (bsvw)=a(csuw).\]
    Thus, $s^2uvw \in \overline{aA}$.
    As $s^2uvw\in \wA A$, the subtractive ideal generated by $a$ contains a semi-invertible element, hence $\overline{aA}=A$ because of \cref{rem: subtractive ideals with semi-invertble elements are full semiring}.
    We conclude that $a$ is semi-invertible.

    Since the image of a semi-invertible element by a homomorphism of semirings is again semi-invertible, all other claims are deduced from the universal property of the localization.
\end{proof}

\begin{example}\label{ex: hardening of B[x]}
    Recall from \cref{ex: hard and non-hard semirings} that the semiring $\mathbb{B}[x]$ is not hard.
    Its hardening satisfies
    \begin{equation}\label{eq: hardening of B[x]}
    \mathbb{B}[x]^\diamondsuit \cong
    \big(\mathbb{N}_{\min}\times\mathbb{Z}_{\max})\cup\{(+\infty,-\infty)\},
    \end{equation}
    where the semiring on the right-hand side has addition given by $(n,d)+(m,e)=(\min(n,m),\max(d,e))$ and multiplication by $(n,d)\cdot(m,e)=(n+m,d+e)$ for all $n,m\in\mathbb{N}$ and~$d,e\in\mathbb{Z}$, and these operations are naturally extended to the point~$\{(+\infty,-\infty)\}$.

    An explicit isomorphism in \eqref{eq: hardening of B[x]} is provided by the map
    \[
    \frac{f}{g}\mapsto\big(\textrm{ord}_0f,\deg f-\deg g\big).
    \]
    This is easily verified to be a well-defined surjective semiring morphism; notice that the fact that it preserves addition follows from the observation that all semi-invertible elements of $\mathbb{B}[x]$ have order of vanishing $0$ at~$0$, by their explicit description in \cref{ex: non-hard semiring} or \cref{lem: semi-invertibles in polynomial ring}.
    
    Let us then prove injectivity.
    Since there is a unique element mapping to~$(+\infty,-\infty)$, we can assume that $f_1/g_1$ and $f_2/g_2$ have the same image in~$\mathbb{N}\times\mathbb{Z}$.
    On the one hand, $f_1$ and $f_2$ have the same order of vanishing at~$0$, so that there exists $k\in\mathbb{N}$ and $h_1,h_2\in\mathbb{B}[x]$ with
    \[
    f_1=x^k(1+h_1)\quad\text{and}\quad f_2=x^k(1+h_2).
    \]
    On the other hand, the equality on the difference of degrees implies that
    \[
    \deg(1+h_1)-\deg g_1
    =
    \deg (1+h_2)-\deg g_2.
    \]
    Then, for $d$ large enough the semi-invertible element $u=1+x+\ldots+x^{d}$ of $\mathbb{B}[x]$ satisfies
    \[
    u(1+h_1)g_2=u(1+h_2)g_1,
    \]
    from which~$uf_1g_2=uf_2g_1$.
    Then, the two fractions coincide.
\end{example}

Many more examples of hard semirings are obtained by localization at primer kernels.

\begin{lemma}\label{lem: localization at prime kernel is hard}
    Let $\p$ be a prime kernel of a semiring~$A$, and let $A_{\p}$ be the corresponding localization.
    Then $A_{\p}$ is a hard semiring.
\end{lemma}
\begin{proof}
    We follow the same argument as in \cref{prop: hardening is hard}.
    Suppose we have a semi-invertible element~$a/s$ of~$A_{\p}$, and we want to show that~$a\notin \p$. Since~$a/s$ is semi-invertible,  there exist $b,c\in A$ and $u,v, w\notin\p$, such that 
    $s^2uvw + a (bsvw)=a(csuw)$.
    This implies that $s^2uvw$ belongs to the kernel generated by~$a$.
    But $s^2uvw\notin \p$, hence $a\notin \p$ since $\p$ is subtractive. 
\end{proof}

\section{Spectrum constructions}\label{sec: spectrum constructions}

In this section, we consider two types of spectra associated to a semiring, and equip them with a suitable topology, emulating the usual construction in algebraic geometry.

Scheme theory corresponding to the spectrum using prime ideals has been carried out in several articles, including \cite{LORSCHEID20121804}, \cite[Section 3]{Giansiracusa2X_2016} and \cite[Section 2]{JUN2017306}, while the $k$-spectrum (or saturated spectrum) as a topological space was considered in \cite{LESCOT20121004}.

Following the common practice, we use gothic letters to denote prime ideals of a semiring.

\subsection{Spectrum and $k$-spectrum}

Let us start with the set-theoretical definitions.

\begin{definition}\label{def: spectra}
    Let $A$ be a semiring.
    The \emph{spectrum of $A$} is the set
    \[
    \Spec A=\{\mathfrak{p}\subseteq A\mid \mathfrak{p} \text{ is a prime ideal}\}.
    \]
    The \emph{$k$-spectrum} (or, the \emph{saturated spectrum}) \emph{of $A$} is
    \[
    \Sp A=\{\mathfrak{p}\subseteq A\mid \mathfrak{p} \text{ is a prime subtractive ideal}\}.
    \]
\end{definition}

Note that when $A$ is a ring, these two notions coincide, by \cref{rem: subtractive ideals in rings}.
In general, we have that~$\Sp A\subseteq\Spec A$, and the containment can be strict, as we next show.

\begin{example}\label{ex: Sp and Spec of N}
    Let us determine the spectrum and the $k$-spectrum of the semiring of natural numbers~$\mathbb{N}$, by listing all of its prime ideals.

    Since $\{0\}$ is obviously a prime ideal, assume $\mathfrak{p}\subseteq\mathbb{N}$ is a nontrivial prime ideal, and let $p$ be the minimal nonzero element of~$\mathfrak{p}$: it must be prime, since otherwise $\mathfrak{p}$ would contain a proper prime factor of~$p$, contradicting its minimality.
    In particular, it can happen that~$\mathfrak{p}=p\mathbb{N}$, which is actually easily verified to be a prime ideal.
    Assume this is not the case; therefore $\mathfrak{p}$ contains an integer that is not a multiple of~$p$.
    Using primality, we deduce that $\mathfrak{p}$ actually contains a prime~$q\neq p$. 
    Then, by additivity,
    \[
    \mathfrak{p}\supseteq p\mathbb{N}\cup (q+p\mathbb{N})\cup (2q+p\mathbb{N})\cup\ldots\cup\big((p-1)q+p\mathbb{N}\big).
    \]
    Since $p$ and $q$ are coprime, the set $\{0,q,2q,\ldots,(p-1)q\}$ exhausts the residue classes modulo~$p$, thus the ideal $\mathfrak{p}$ contains all integers larger than~$(p-1)q$.
    In particular, for all $n\neq 1$ we can take a sufficiently large power of $n$ that belongs to $\mathfrak{p}$ and apply primality to deduce that~$n\in\mathfrak{p}$.
    We conclude that~$\mathfrak{p}\supseteq\mathbb{N}\setminus\{1\}$.
     
    Combining with \cref{ex: a prime ideal of N}, we conclude that
    \[
    \Spec \mathbb{N}=\{\{0\}\}\cup\{p\mathbb{N}\mid p\text{ prime}\}\cup\{\mathbb{N}\setminus\{1\}\}.
    \]
    Let us now determine the $k$-spectrum of~$\mathbb{N}$.
    The ideal $\{0\}$ is obviously subtractive.
    It follows from \cref{prop: k-ideals} that ideals of the form $p\mathbb{N}$ are also subtractive, since they are the kernel of the natural map $\mathbb{N}\to\mathbb{Z}/p\mathbb{Z}$ taking any natural number to its class modulo~$p$.
    On the contrary, the ideal $\mathbb{N}\setminus\{1\}$ is not subtractive, as follows from \cref{rem: relation between kernel and quotient} and \cref{prop: k-ideals}.
    Hence, we conclude that
    \[
    \Sp \mathbb{N}=\{\{0\}\}\cup\{p\mathbb{N}\mid p\text{ prime}\}.
    \]
\end{example}

\begin{remark}
    Let $F$ be a semifield.
    By \cref{rem: ideals of semifields}, we have
    \[
    \Spec F=\Sp F=\{\{0_F\}\}.
    \]
    This is the case, for instance, when $F$ is the boolean domain $\mathbb{B}$ or the tropical semiring $\mathbb{T}$ from \cref{ex: semirings}.
\end{remark}

The difference between the spectrum and the $k$-spectrum can be even more drastic than in \cref{ex: Sp and Spec of N}.
To see this, we prove the following result, of independent interest, which realizes the $k$-spectrum of an idempotent semiring as the set of its morphisms to the boolean domain.

\begin{proposition}\label{lem: Sp of idempotent semiring}
    If $A$ is idempotent, there is a bijection $\Sp A\cong \operatorname{Hom} (A,\mathbb{B})$.
\end{proposition}
\begin{proof}
    The ideal $\{0\}$ is prime in the boolean domain~$\mathbb{B}$.
    Therefore, for any morphism~$f\colon A\to \mathbb{B}$, the kernel of $f$ is a prime subtractive ideal of $A$ because of \cref{lem: preimage of ideals} and \cref{prop: k-ideals}. 
    
    Conversely, for a given $\p\in \Sp A$ define a map $A\to\mathbb{B}$ by
    \[
    a\longmapsto
    \begin{cases}
    0 &\text{if }a\in \p,
    \\1&\text{otherwise}.
    \end{cases}
    \]
    It is a homomorphism of semirings; the fact that it preserves addition follows from the assumptions that $\p$ is a subtractive ideal and $A$ is idempotent. 
\end{proof}


\begin{corollary}\label{cor: Sp for polynomials}
Let $F$ be an idempotent semifield.
Then
\[
\Sp F[x_1,\ldots,x_n]
=
\{\langle x_j\rangle_{j\in J}\mid J\subseteq\{1,\ldots,n\}\}.
\]
\end{corollary}
\begin{proof}
Since $F$ is idempotent and as seen in \cref{ex: semirings}, the semiring $F[x_1,\ldots,x_n]$ is also idempotent.
Therefore, because of the proof of \cref{lem: Sp of idempotent semiring} there is a bijection
\[
\operatorname{Hom} (F[x_1,\ldots,x_n],\mathbb{B})
\cong\Sp F[x_1,\ldots,x_n]
\]
identifying each morphism $f\colon F[x_1,\ldots,x_n]\to\mathbb{B}$ with its kernel.
Since any morphism of semirings sends invertible elements to invertible elements and $F$ is a semifield, $f$ maps every nonzero element of $F$ to~$1$.
As a consequence, $\ker(f)$ is seen to agree with the ideal generated by the set~$\{x_j\mid f(x_j)=0\}$.

Moreover, by the same reason a morphism from $F[x_1,\ldots,x_n]$ to $\mathbb{B}$ is equivalent to the choice of an image in $\mathbb{B}$ for each of the variables or, equivalently, to the choice of a subset of~$\{1,\ldots,n\}$.
\end{proof}

Let $F$ be an idempotent semifield.
In view of \cref{cor: Sp for polynomials}, mapping a subset $J\subseteq\{1,\ldots,n\}$ to the ideal generated by $\{x_j\mid j\in J\}$ realizes a bijection between the power set of~$\{1,\ldots,n\}$ and~$\Sp F[x_1,\ldots,x_n]$, and this identification respects inclusions.

In particular, $\Sp F[x_1,\ldots,x_n]$ is a finite set of cardinality~$2^n$.

\begin{example}\label{ex: Sp and Spec of B[x]}
    Thanks to \cref{cor: Sp for polynomials} we have that~$\Sp \mathbb{B}[x]=\{\{0\},\langle x\rangle\}$.
    Instead, $\Spec \mathbb{B}[x]$ is infinite and it actually contains infinitely increasing sequences of prime ideals, see \cite[Corollary 12]{alarcon1994commutative}.

For more results on the $k$-spectrum in the case of idempotent semirings, the reader can consult \cite{LESCOT20121004}. 

\end{example}

\subsection{The Zariski topology}
 
Let $A$ be a semiring.
We now define a topology on the spectra of $A$ from \cref{def: spectra} as done in classical scheme theory.
To do so, for a subset $S\subseteq A$ we set
\[
V(S) = \{\mathfrak{p} \in\Spec A \mid S \subseteq  \mathfrak{p}\}.
\]
We have the following properties.

\begin{lemma}\label{lem: closed subsets of Spec}
    In the above notation:
    \begin{enumerate}
        \item for a family $(S_i)_{i\in I}$ of subsets of $A$ we have $\bigcap_{i\in I}V(S_i)= V\big(\bigcup_{i\in I} S_i\big)$;
        \item for all $S_1,S_2\subseteq A$ we have $V(S_1)\cup V(S_2)= V(S_1 \cdot S_2)$;
        \item $V(\emptyset)= \Spec A$ and $V(A)= \emptyset$.
    \end{enumerate}
\end{lemma}
\begin{proof}
    The first and third assertions are obvious.
    For the second, it is clear that if $S_1\subseteq\mathfrak{p}$ or $S_2\subseteq\mathfrak{p}$ then $S_1S_2\subseteq\mathfrak{p}$.
    Conversely, assume that $S_1S_2\subseteq\mathfrak{p}$ and suppose, for example, that~$\mathfrak{p}\notin V(S_1)$.
    Then, arguing as in \cite[Lemma~2.1]{Hartshorne}, there is $s_1\in S_1$ with~$s_1\notin\mathfrak{p}_1$.
    For all $s_2\in S_2$ we have that $s_1 s_2\in\mathfrak{p}$ and so by primality $s_2\in\mathfrak{p}$; this implies that $S_2\subseteq\mathfrak{p}$.
\end{proof}

In particular, \cref{lem: closed subsets of Spec} ensures that the collection
\begin{equation}\label{eq: closed set of Spec}
    \{V(S)\mid S\subseteq A\}
\end{equation}
satisfies the axioms of the closed sets of a topology on~$\Spec A$.

\begin{definition}
    The \emph{Zariski topology on $\Spec A$} is the topology whose closed sets are those of the form $V(S)$ as in \eqref{eq: closed set of Spec}.
    Similarly, the subset topology of $\Sp A\subseteq\Spec A$ is called the \emph{Zariski topology on~$\Sp A$}.
\end{definition}

Throughout these notes, the sets $\Spec A$ and $\Sp A$ will always be considered as topological spaces equipped with their Zariski topology. 

\begin{remark}
Let $A$ be an idempotent semiring.
We can define a topology on the set $\Hom(A,\mathbb{B})$ by declaring that a subset is closed if and only if it is of the form
\[
\{f\in\Hom(A,\mathbb{B})\mid f(a)\leq g(a)\text{ for all }a\in A\}
\]
for some arbitrary function~$g\colon A\to\mathbb{B}$, that is not necessarily a semiring homomorphism.
This choice of the topology on $\Hom(A,\mathbb{B})$ turns the bijection of \cref{lem: Sp of idempotent semiring} into a homeomorphism. 
\end{remark}


\begin{remark}\label{rem: dimension of semiring spectrum}
    The topological spaces $\Spec A$ and~$\Sp A$, despite being defined similarly, can be quite different.
    
    For example, let us consider the semiring $\mathbb{N}$ of natural numbers and denote by $\mathfrak{m}=\mathbb{N}\setminus\{1\}$ its unique maximal prime ideal.
    Thanks to the explicit description in \cref{ex: Sp and Spec of N} and in view of the definition of the Zariski topology, the only irreducible closed subsets of $\Spec \mathbb{N}$ are the singleton~$\{\mathfrak{m}\}$, the subsets of the form $\{p\mathbb{N},\mathfrak{m}\}$ for some prime~$p$, and the entire~$\Spec\mathbb{N}$.
    It follows that
    \[
    \dim\Spec \mathbb{N}=2,
    \]
    where $\dim$ stands for the topological dimension.
    On the contrary, there is a canonical homeomorphism between $\Sp \mathbb{N}$ and the usual affine scheme~$\Spec\mathbb{Z}$, and so
    \[
    \dim \Sp \mathbb{N}=1.
    \]

    The situation is even more extreme for the semiring~$\mathbb{B}[x]$.
    Indeed, it follows from \cref{ex: Sp and Spec of B[x]} that $\Sp\mathbb{B}[x]$ has topological dimension~$1$, while $\Spec\mathbb{B}[x]$ has infinite topological dimension.
\end{remark}

As in usual algebraic geometry, and in particular to define a structure sheaf on the spaces $\Spec A$ and~$\Sp A$, it is useful to have a basis of open sets for their Zariski topology.
To this end, for each $a\in A$ we consider the \emph{principal open set}
\[
D(a) = \{\p \in \Spec A \mid a \notin \p  \} = \Spec A \setminus V(\{a\}).
\]
The importance of sets of this kind is due to the following property.

\begin{lemma}\label{lem: principal open sets are basis for Zariski topology of Spec}
    The family $\{D(a)\mid a\in A\}$ is a basis for the Zariski topology of $\Spec A$ and it satisfies
    \[
    D(ab)=D(a)\cap D(b)\quad\text{ and }\quad D(a+b) \subseteq D(a)\cup D(b)
    \]
    for every~$a,b\in A$.  
    Moreover, a subset $S\subseteq A$ gives a covering of $\Spec A$ by the principal open sets $\{D(a)\}_{a\in S}$, that is \[\Spec A= \bigcup_{a\in S} D(a),\]
    if and only if the ideal generated by $S$ is equal to~$A$.
    In particular:
    \begin{enumerate}
    \item $D(a)=\Spec A$ if and only if $a$ is invertible;
    \item for any covering $\Spec A=\bigcup_{a\in S} D(a)$ by principal open sets, there exists a finite subset $T\subseteq S$ such that $\Spec A=\bigcup_{a\in T} D(a)$.
    \end{enumerate}
    
\end{lemma}
\begin{proof}
        The sets in the family are open by definition.
        To show they form a basis for the Zariski topology we have to show that for all closed subset $Z\subseteq\Spec A$ and for any point $\mathfrak{p}\notin Z$ there exists $a\in A$ such that $\mathfrak{p}\in D(a)$ and~$D(a)\cap Z=\emptyset$.
        By definition, it happens that $Z=V(S)$ for a subset $S\subseteq A$ and $S\nsubseteq \mathfrak{p}$; it is then sufficient to take $a\in S\setminus\mathfrak{p}$ to conclude.
        
        The facts that $D(ab)=D(a)\cap D(b)$ and $D(a+b) \subseteq D(a)\cup D(b)$ are immediate from the definition and \cref{lem: closed subsets of Spec}.

        If $S$ generates~$A$, then there is no prime ideal containing $S$, which means that~$V(S)=\emptyset$.
        It follows from \cref{lem: closed subsets of Spec}.(1) that $\Spec A=\bigcup_{a\in S} D(a)$.

        Conversely, by the well-known consequence of Zorn's lemma that any proper ideal of $A$ is contained in a maximal (so, prime) ideal of~$A$, if $S$ does not generate~$A$, then there exists a prime ideal $\mathfrak{p}\in \Spec A$ that contains~$S$.
        So $\mathfrak{p}\notin D(a)$ for all~$a\in S$, and therefore $\Spec A\neq  \bigcup_{a\in S} D(a)$.
        
        By the previous assertion, $D(a)=\Spec A $ is equivalent to the equality~$aA=A$, and then to the fact that $a$ is invertible because of \cref{rem: ideal generated by invertibles}.

        Finally, notice that if $S\subseteq A$ generates the unit ideal, then there exists a finite subset $T\subseteq S$ that generates the unit ideal: as $1_A$ is in the ideal generated by $S$, there exists $a_i\in S$ and $b_i\in A$ for $i=1,\dots, n$ such that $1_A=a_1b_1+\cdots a_nb_n$. Then we can take $T=\{a_1,\dots,a_n\}$. 
\end{proof}

Similarly, for a given element $a\in A$, we will denote
\[
\bD(a) = D(a)\cap \Sp A=\{\p \in \Sp A \mid a \notin \p  \}.
\]
The $k$-spectral version of \cref{lem: principal open sets are basis for Zariski topology of Spec} reads as follows. 

\begin{lemma}\label{lem: principal open sets are basis for Zariski topology of Sp}
    The family $\{\bD(a)\mid a\in A\}$ is a basis for the Zariski topology of~$\Sp A$ and it satisfies
    \[
    \bD(ab)=\bD(a)\cap \bD(b)\quad\text{ and }\quad\bD(a+b) \subseteq \bD(a)\cup \bD(b)
    \]
    for every~$a,b\in A$. Moreover, if $b$ is in the subtractive ideal generated by $a$, then $\bD(b)\subseteq \bD(a)$.
    
    Finally, a subset $S\subseteq A$ gives a covering of $\Sp A$ by the open principal sets $\{\bD(a)\}_{a\in S}$
    if and only if the subtractive ideal generated by $S$ is equal to $A$.
    In particular:
    \begin{enumerate}
    \item $\bD(a)=\Sp A$ if and only if $a$ is semi-invertible;
    \item for any covering $\Sp A=\bigcup_{a\in S} \bD(a)$ by principal open sets, there exists a finite subset $T\subseteq S$ such that $\Sp A=\bigcup_{a\in T} \bD(a)$.
    \end{enumerate}
\end{lemma}
\begin{proof}
    The first three assertions are a direct consequence of the corresponding properties in \cref{lem: principal open sets are basis for Zariski topology of Spec} and the definition of the topology of~$\Sp A$ as the subspace topology.

    If $b$ is in the subtractive ideal generated by $a$, and $a\in  \p\in \Sp A$, then $b\in \p$, as $\p$ is subtractive, which shows that $\bD(b)\subseteq \bD(a)$.

    The rest of the claims are shown using the same arguments as in the ideal case.
    In particular, similarly to the proof of \cref{lem: principal open sets are basis for Zariski topology of Spec}, one needs to show that any proper subtractive ideal is contained in a maximal subtractive ideal, hence in a prime subtractive ideal, by using again Zorn's lemma.
\end{proof}

Finally, for idempotent semirings we can use \cref{lem: Sp of idempotent semiring} to give a description of the principal open sets in its $k$-spectrum.  

\begin{lemma}\label{lem: principal opens for idempotents}
Let $A$ be an idempotent semiring. 
Given any $a\in A$, the natural bijection $\Sp A\cong \Hom(A,\mathbb{B})$ induces a bijection \[\bD(a) \cong \{f\colon A\to \mathbb{B}\mid f(a)=1 \}.\]

    Moreover, for any $a,b\in A$, we have $\bD(a)\cup \bD(b)=\bD(a+b)$.
\end{lemma}
\begin{proof}
Recall that the bijection $\Sp A\cong \Hom(A,\mathbb{B})$ from \cref{lem: Sp of idempotent semiring} identifies a morphism $f\colon A\to \mathbb{B}$ with its kernel.
Therefore, the prime ideal corresponding to $f$ lies in $\bD(a)$ if and only if $f(a)\neq 0$, equivalently~$f(a)=1$.  

    The last claim is deduced from the fact that $f(a+b)=f(a)+f(b)=1$ if and only if at least one between $f(a)$ and $f(b)$ is equal to~$1$. 
\end{proof}

\subsection{Induced morphisms}

Let us conclude the section with the construction of morphisms between semiring spectra from their algebraic counterpart.
Let $f\colon A \to B$ be a morphism of semirings, and consider the map
\[
\Spec(f)\colon \Spec B \longrightarrow \Spec A
\]
defined by mapping a prime ideal $\mathfrak{p}$ of $B$ to $f^{-1}(\mathfrak{p})$.
This map is well-defined thanks to \cref{lem: preimage of ideals} and it restricts to a map
\[
\Sp(f)\colon \Sp B \longrightarrow \Sp A
\]
because of \cref{lem: preimage of subtractive ideal}.

\begin{lemma}\label{lem: continuity of morphisms}
    The maps $\Spec(f)$ and $\Sp(f)$ are continuous.
\end{lemma}
\begin{proof}
    Since the restriction of a continuous map to a subset equipped with the subset topology is again continuous, it is enough to check the continuity of~$\Spec(f)$.
    To show the latter, let $S$ be any subset of $A$ and notice that
    \[
    \big(\Spec(f)\big)^{-1}(V(S))
    =
    \big\{\mathfrak{p}\in\Spec B\mid S\subseteq f^{-1}(\mathfrak{p})\big\}
    =
    V(f(S)).
    \]
    In other words, the preimage of a closed set is closed, and $\Spec(f)$ is continuous.
\end{proof}

In particular, $\Spec$ and $\Sp$ yield contravariant functors from the category of semirings to the one of topological spaces.

We can apply \cref{lem: continuity of morphisms} to two cases of special interest.
To this end, recall that for a multiplicative submonoid $S$ of a semiring $A$ we denote by $\varphi_S\colon A\to S^{-1}A$ the corresponding localization homomorphism.

\begin{lemma}\label{lem: spectral morphism associated to localization}
    For any semiring $A$ and a multiplicative submonoid $S$ of~$A$, the map $\Spec(\varphi_S)$ induces a homeomorphism
    \[\Spec S^{-1}A \cong\{\mathfrak{p}\in\Spec A\mid \mathfrak{p}\cap S=\emptyset\}.\]
    The same property holds for $\Sp$.
\end{lemma}
\begin{proof}
    The continuity of $\Spec(\varphi_S)$ follows from \cref{lem: continuity of morphisms}.
    Moreover, for all $a\in A$ and $s\in S$ it is easily checked that a prime ideal of $S^{-1}A$ contains $a/s$ if and only if it contains~$\varphi_S(a)$.
    Then,
    \begin{equation}\label{eq: image of principal open by localization}
        \Spec(\varphi_S)(D(a/s))=D(a)
    \end{equation}
    and so the map is open.
    It is also clear that it is injective and that its image is $\{\mathfrak{p}\in\Spec A\mid \mathfrak{p}\cap S=\emptyset\}$, see for instance \cite[Proposition 11.15]{Golan1999}.
    The proof works the same for~$\Sp$.
\end{proof}

\begin{lemma}\label{lem: hardering on Sp}
    For a semiring $A$, consider the hardening morphism $\chi\colon A\to A^\diamondsuit$. Then the map $\Sp(\chi)\colon\Sp A^\diamondsuit \to \Sp A$ is a homeomorphism. 
\end{lemma}
\begin{proof}
    Applying \cref{lem: spectral morphism associated to localization} in our case, we get a homeomorphism \[\Sp A^\diamondsuit\cong \{\p \in \Sp A\mid \p \cap \wA A=\emptyset\},\]
    where $S(A)$ is the multiplicative submonoid of semi-invertible elements of $A$ from \cref{lem: semiinvertibles are saturated multiplicative monoid}.
    Now, the set on the right-hand side agree with $\Sp A$ since a proper subtractive ideal of $A$ cannot contain any semi-invertible element by \cref{rem: subtractive ideals with semi-invertble elements are full semiring}.
\end{proof}

\section{The structure sheaf for \texorpdfstring{$\Spec$}{}}\label{sec: structure sheaf for Spec}

We fix for the entire section a semiring $A$ and we let $X = \Spec A$ be the associated spectrum, seen as a topological space like in \cref{sec: spectrum constructions}.

The goal of this section is to define a suitable \emph{structure sheaf} $\mathcal{O}_X$ on~$X$.
In fact, we give three equivalent definition for such a sheaf.
Their analogue for~$\Sp A$ will be studied in \cref{sec: structure sheaf for Sp}.

Throughout, recall that given a prime ideal $\p \subseteq A$ we denote by $A_\p = ( A \setminus \p)^{-1}A$ the corresponding localization.

\subsection{Three constructions of $\mathcal{O}_X$}\label{subsec: three constructions of the structure sheaf on Spec}

Given any open subset $U\subseteq X$, we set
\begin{equation}\label{eq: multiplicative submonoid associate to open U}
    S_U = \{b\in A\mid D(b)\supseteq U \}\subseteq A.
\end{equation}
As in usual algebraic geometry, this can be seen as the subset of elements of $A$ which, thought of as functions on~$X$, never vanish on~$U$.
In the particular case in which $U=D(a)$ for some $a\in A$ we write for short\[S_a=S_{D(a)}.\]

The next lemma describes the general algebraic structure of this set and makes it explicit in the case in which $U$ is a principal open set.

\begin{lemma}\label{lem: submonoid of open set}
For any open subset~$U\subseteq X$, the set $S_U$ is a saturated multiplicative submonoid of~$A$.
Moreover, for all $a\in A$, we have that
\[S_{a}=(a^\mathbb{N})^{\mathrm{sat}}=\{b\in A \mid \exists c\in A \text{ and } \exists n\in \mathbb{N} \text{ with } bc = a^n \} \] 
is the saturation of the multiplicative submonoid generated by~$a$.
\end{lemma}
\begin{proof}
    First, notice that $1_A\in S_U$, since $D(1_A)=X\supseteq U$.
    We need to show that for all $a,b\in A$ one has $ab\in S_U$ if and only if $a\in S_U$ and~$b\in S_U$.
    Using the equality $D(ab)=D(a)\cap D(b)$ from \cref{lem: principal open sets are basis for Zariski topology of Spec}, this is equivalent to
    \[D(a)\cap D(b) \supseteq U \iff D(a)\supseteq U \text{ and } D(b)\supseteq U,\]
    which is clear. 

    To show the last claim, first observe that $a\in S_a$ clearly.
    Since $S_a$ is a saturated multiplicative submonoid, $a^{\mathbb{N}}\subseteq S_a$ and so $(a^\mathbb{N})^{\mathrm{sat}}\subseteq S_a$ by minimality of the saturation, see \cref{lem: Saturation}.

    In order to show the reverse inclusion, take $b\in S_a$, so that $D(b)\supseteq D(a)$. This is equivalent to say that  the  prime ideals of $A$ that contain $b$ also contain~$a$, or, equivalently, 
    that \[ a\in \bigcap_{\substack{b\in\p\\\p\in \Spec A}} \p = \rad(bA),\]
    where the last equality is \cref{thm: radical of ideals}. Therefore, there exists $n\ge 0$ and $c\in A$ such that $a^n=bc$, so $b$ is in the saturation of $a^\mathbb{N}$. 
    \end{proof}

Notice that the localization of $A$ at the multiplicative submonoid generated by $a$ is what is usually called~$A_a=A[a^{-1}]$.
So, \cref{lem: submonoid of open set} implies that $S_a^{-1}A\cong A_a$ functorially. 

The next result shows that principal open sets are in fact affine open.

\begin{lemma}\label{lem: principal open are affine}
    For any $a\in A$, let $\varphi_a\colon A \to S_a^{-1}A$ be the localization map.
    Then the map $\Spec(\varphi_a)$ induces a homeomorphism $\Spec S_a^{-1}A \cong D(a)$.
\end{lemma}
\begin{proof}
    By \cref{lem: spectral morphism associated to localization} we already know that $\Spec(\varphi_a)$ induces a homeomorphism between $\Spec S_a^{-1}A$ and the set of primes $\p\in X$ such that~$\p\cap S_a=\emptyset$.
    But notice that if $b\in \p\cap S_a$, then by \cref{lem: submonoid of open set} there exists $c\in A$ such that $bc=a^n$ for some $n\in\mathbb{N}$.
    Therefore $a^n=bc\in \p$, so~$a\in\p$.
    Hence $\p\cap S_a=\emptyset$ if and only if~$a\notin\p$. 
\end{proof}

We are now ready to present three independent approaches to the construction of the structure sheaf $\mathcal{O}_X$ on~$X$.
To this end, let us recall that a presheaf $\mathcal{F}$ (of semirings) on a topological space is a sheaf if for any covering of an open sets $U$ by open sets $U_i$ for $i\in I$, that is $U=\bigcup_{i\in I} U_i$, the map 
\[\mathcal{F}(U)\longrightarrow \operatorname{Eq}\bigg(\prod_{i\in I} \mathcal{F}(U_i) \rightrightarrows \prod _{i,j\in I} \mathcal{F}(U_i\cap U_j)\bigg)\]
induced by restrictions is an isomorphism.
Here $\operatorname{Eq}$ denotes the equalizer: as a set, is formed by the elements $(f_i)_i\in \prod_{i\in I} \mathcal{F}(U_i)$ such that $f_i=f_j \in \mathcal{F}(U_i\cap U_j)$ for all~$i,j\in I$.

\subsection*{Approach A: Sheafification of the presheaf of localizations}

Consider the presheaf of semirings $\mathcal{F}$ on $X$ defined by
\begin{equation}\label{eq: presheaf in Approach A}
    \mathcal{F}(U)=S_U^{-1} A
\end{equation}
for all open subset~$U\subseteq X$, where $S_U$ is the saturated multiplicative submonoid associated to~$U$, as defined in \eqref{eq: multiplicative submonoid associate to open U}.
Intuitively, the semiring $\mathcal{F}(U)$ is obtained by inverting all functions which are nowhere zero on~$U$.

The presheaf $\mathcal{F}$ is not a sheaf in general (even for commutative rings, see \cref{ex: presheaf of localization can fail to be a sheaf} below), so we take the structure sheaf
\[
\mathcal{O}_X= \mathcal{F}^\#
\]
to be its sheafification.

\begin{example}\label{ex: presheaf of localization can fail to be a sheaf}
    Let $K$ be an algebraically closed field of characteristic~$0$, and let $A=K[t^2,t^3]$ be the subring of $K[t]$ consisting of the polynomials in $t$ with degree-one coefficient equal to~$0$.
    Then, the presheaf $\mathcal{F}$ defined above is not a sheaf on the spectrum of this ring.
    
    To show this, consider the point $\mathfrak{m}\in\Spec A$ given by the kernel of the evaluation morphism~$f\mapsto f(1)$.
    It is a maximal ideal and explicitly one has~$\mathfrak{m}=(t^2-1)A+(t^3-1)A$.
    Consider the open set $U=\Spec A \setminus \{\mathfrak{m}\}$. 

    Observe that~$S_U=K^{\times}$, as any polynomial $f\in S_U$ must have its set of roots inside~$\{1\}$, hence it must be of the form $f=\lambda (t-1)^n$ for some $n\ge 0$ and some $\lambda \in K^{\times}$.
    But if~$n>0$, then the degree-one coefficient of such $f$ is $\pm \lambda n\ne 0$, so $f$ would not belong to~$A$.

    Now, consider the covering $U=D(t^2-1)\cup D(t^3-1)$.
    We will give an element in the equalizer of the map 
    \[S^{-1}_{t^2-1} A \times S^{-1}_{t^3-1} A \rightrightarrows S^{-1}_{(t^2-1)(t^3-1)} A \]
    that is not coming form an element $f\in S_U^{-1}A=A$. The element is 
    \[ \left(\frac{t^3+t^2}{t^2-1},\frac{t^4+t^3+t^2}{t^3-1} \right).\]
    It is clear it is in the equalizer, as 
    \[\frac{t^3+t^2}{t^2-1}=\frac{t^4+t^3+t^2}{t^3-1}\in S^{-1}_{(t^2-1)(t^3-1)} A \]
    as both are equal to $t^2/(t-1) \in K(t)$. And for the same reason, none of them can be of the form $f/1$ for $f\in A$ in any of the local rings. 
\end{example}

\subsection*{Approach B: Extension of the sheaf of localizations defined on principal opens}

This approach consists in first constructing a sheaf on a basis for the Zariski topology of $X$ and then extending it to a sheaf on the whole spectrum, as explained in \cite[Section 6.30]{stacks-project}.

To this end, we consider the basis for the topology of $X$ given by principal open sets, as seen in \cref{lem: principal open sets are basis for Zariski topology of Spec}.
The presheaf $\mathcal{F}$ defined in \eqref{eq: presheaf in Approach A} has local sections on this basis given by 
\begin{equation}\label{eq: definition of sheaf for Spec on principal open sets}
    \mathcal{F}(D(a)) = S_{a}^{-1}A.
\end{equation}
for all $a\in A$.

As we will see in \cref{lem: SheafLemma} below, the presheaf $\mathcal{F}$ is actually a sheaf on the basis of principal open sets, and it can then be extended in a unique and functorial way to a sheaf $\mathcal{O}_X$ on $X$ by setting, for all open $U\subseteq X$ 
\begin{equation}\label{eq: sheaf from Approach B}
    \mathcal{O}_X(U)=\operatorname{Eq}\bigg(\prod_{i\in I} \mathcal{F}(D(a_i)) \rightrightarrows \prod _{i,j\in I} \mathcal{F}(D(a_ia_j))\bigg)
\end{equation}
for any covering $U=\bigcup_{i\in I} D(a_i)$.
The resulting semiring $\mathcal{O}_X(U)$ is independent of the chosen covering, by the sheaf property.
In particular, one has $\mathcal{O}_X(D(a))=\mathcal{F}(D(a))$ for all~$a\in A$ by construction \cite[Lemma~6.30.6]{stacks-project}.

\begin{remark}
Note that the assignment in \eqref{eq: definition of sheaf for Spec on principal open sets} is more suitable for formalization than declaring $\mathcal{F}(D(a)) = A_{a}$.
Indeed, this last choice would depend on the choice of $a$ rather than on~$D(a)$, in the sense that if $D(a)=D(b)$ then $A_a$ and $A_b$ are equal only up to an isomorphism of semirings, while by definition~$S_a=S_b$.
\end{remark}

Let us now show in details that the presheaf $\mathcal{F}$ is actually a sheaf on the basis of principal open sets of~$X$, in the sense of \cite[Definition~6.30.2]{stacks-project}.
This amounts to show the following statement.

\begin{lemma}\label{lem: SheafLemma}
    Given $a\in A$ and  $a_i \in A$ for $i\in I$ such that $\bigcup_{i\in I} D(a_i) = D(a)$, the morphism
    \[S^{-1}_{a} A \longrightarrow\operatorname{Eq}\bigg(\prod_{i\in I} S^{-1}_{a_i} A \rightrightarrows \prod _{i,j\in I}  S^{-1}_{a_ia_j} A\bigg)\]
    is an isomorphism.
\end{lemma}

We will divide the proof in several sublemmas.
We start by showing the injectivity in \cref{lem: SheafLemma} in the special case in which~$D(a)=X$.

\begin{sublemma}\label{subl: injectivity global}
    Given $a_i \in A$ for $i\in I$ such that $\bigcup_{i\in I} D(a_i) =  \Spec A$, the morphism 
    \[A \longrightarrow \prod_{i\in I} S^{-1}_{a_i} A\]
    is injective.     
\end{sublemma}
\begin{proof}
    Take $x,y \in A$ with $x/1_A=y/1_A$ in $S_{a_i}^{-1}A$ for every~$i$.
    Then, for all $i\in I$ there exist $u_i\in S_{a_i} $ such that~$x u_i = y u_i$. 
    
    Since $D(u_i)\supseteq D(a_i)$ and the opens $D(a_i)$ cover~$\Spec A$, we have that $\bigcup_{i\in I} D(u_i)=\Spec A$.
    Thus, as in the proof of \cref{lem: principal open sets are basis for Zariski topology of Spec} we can choose $b_i\in A$ with $b_i=0_A$ for all but a finite number of $i$'s, such that $1_A=\sum b_i u_i$.
    Then
    \[x= x\sum b_i u_i= \sum b_i x u_i=\sum b_i y u_i=y \sum b_i u_i= y,\]
    concluding the proof.
\end{proof}
  
In order to show the general case, we need to compare the presheaf on a principal open set $D(a)$ and the one on the space~$\Spec S_a^{-1}A $, which are homeomorphic by \cref{lem: principal open are affine}.
This is given by the following technical result. 

\begin{sublemma}\label{subl: localization of localization}
    Given $a,b\in A$ such that $D(a)\supseteq D(b)$, the restriction morphism $S_a^{-1}A \to S_b^{-1}B$ induces an isomorphism \[S^{-1}_{b/1_A} \left(S^{-1}_{a} A\right) \cong S^{-1}_{b} A.\]
\end{sublemma}
\begin{proof} First notice that $D(a)\supseteq D(b)$ implies that $S_a\subseteq S_b$ by definition.
Hence the elements in $S_a$ have invertible image by the localization morphism $A\to S_b^{-1}A$ and therefore by the universal property of localization we have a unique morphism $S_a^{-1}A \to S_b^{-1}A$ which restrict to the localization on~$A$. 

Observe that the equality  
\begin{equation}\label{eq: equation in localization of localization}
    S_{b/1_A}=\left\{\frac{c}{d} \Bigm| c\in S_b \text{ and } d\in S_a\right\}
\end{equation}
of subsets of~$S_a^{-1}A$ holds.
That an element in the right-hand side is in the left-hand side is clear from the explicit description in \cref{lem: submonoid of open set}.
To see the reverse, note that if an element $c/d\in S_a^{-1}A$ belongs to $S_{b/1_A}$ then it verifies
\[D\left(\frac{c}{d}\right)\supseteq D\left(\frac{b}{1_A}\right)\]
by definition, and so $D(c)\supseteq D(b)$ because of \eqref{eq: image of principal open by localization}.
In particular,~$c\in S_b$. 

Now, the fact that every element $c/d\in S_{b/1_A}\subseteq S^{-1}_{a}A$ has invertible image in $S_b^{-1}A$ thanks to \eqref{eq: equation in localization of localization} implies again by the universal property of localization that there is a morphism
\[
S_{b/1_A}^{-1}(S_a^{-1}A) \to S_b^{-1}A.
\]

     The inverse morphism is given by the universal property from the composition of localization maps 
\[A \to S^{-1}_{a} A \to S^{-1}_{b/1_A} \left(S^{-1}_{a} A\right)\]
since an element $c\in A$ has invertible image if and only if $c/1_A\in S_{b/1_A}$, so if and only if $c\in S_b$.

\end{proof}
\begin{sublemma}\label{subl: injectivity principal}
    Given $a\in A$ and  $a_i \in A$ for $i\in I$ such that $\bigcup_{i\in I} D(a_i) = D(a)$, the morphism 
    \[S^{-1}_{a} A \longrightarrow \prod_{i\in I} S^{-1}_{a_i} A\]
    is injective.     
\end{sublemma}
\begin{proof}
As in \cref{lem: principal open are affine} there is a homeomorphism $D(a)\cong\Spec S^{-1}_{a} A $ which identifies $D(a_i)$ with~$D(a_i/1_A)$.
Then, by hypothesis
\[\bigcup_{i\in I} D\left(\frac{a_i}{1_A} \right) = \Spec S^{-1}_{a} A\]
where $a_i/1_A\in S^{-1}_{a} A$.
The claim follows from \cref{subl: injectivity global} applied to the semiring $S^{-1}_{a} A$ and \cref{subl: localization of localization}.
\end{proof}

Coming back to the assumption~$D(a)=X$, we now focus on proving the surjectivity in \cref{lem: SheafLemma}.
We will make use of the following technical result, that allows to represent elements which agree in different localizations through a more desirable form.

\begin{sublemma}\label{PreviousLemma}
    Given $a_1,\ldots,a_n\in A$ and a collection $x_i/s_i\in S^{-1}_{a_i} A$ satisfying $x_i/s_i=x_j/s_j$ in $S^{-1}_{a_i a_j} A$, there exist $x_i'\in A$ and $s'_i \in S_{a_i}$ such that \[\frac{x_i'}{s_i'}=\frac{x_i}{s_i}  \in S^{-1}_{a_i} A\] for every $i=1,\ldots,n$, and $x_i's_j'=x_j's_i'$ for all $i$ and~$j$.
\end{sublemma}
\begin{proof}
By hypothesis, for all pair of indices $i,j$ there exists $u_{ij}\in S_{a_i a_j}$ such that
\begin{equation}\label{eq: equality in previous lemma}
    x_is_j u_{ij} =x_js_i u_{ij}.
\end{equation}
Because of \cref{lem: submonoid of open set}, applied to both $s_i\in S_{a_i}$ and~$u_{ij}\in S_{a_i a_j}$, there exist $v_i,v_{ij}\in A$ and $n_i,n_{ij}\in\mathbb{N}$ such that $s_iv_i=a_i^{n_i}$ and $u_{ij}v_{ij}=(a_{i}a_{j})^{n_{ij}}$.
Therefore, multiplying both sides of \eqref{eq: equality in previous lemma} by $v_iv_jv_{ij}$ we obtain that
\[(x_iv_i) a_j^{n_j} (a_{i}a_{j})^{n_{ij}} =(x_jv_j) a_i^{n_i} (a_{i}a_{j})^{n_{ij}}.\] 
Taking $N= \max_{i,j}n_{ij}$, which is finite since the set of index is finite, and multiplying by a suitable power of $a_ia_j$ both sides of the last equality, we have that
\[(x_iv_ia_i^N) a_j^{N+n_j} = (x_iv_i) a_j^{n_j} (a_{i}a_{j})^{N} =(x_jv_j) a_i^{n_i} (a_{i}a_{j})^{N}=(x_jv_ja_j^N) a_i^{N+n_i}\]
for all~$i,j$.
The claim follows defining $x_i'=x_iv_ia_i^N$ and~$s_i'=a_i^{N+n_i}$. 
\end{proof}

\begin{sublemma}\label{subl: surjectivity global}
    Given $a_i \in A$ for $i\in I$ such that $\bigcup_{i\in I} D(a_i) =  \Spec A$, the morphism
    \[A \longrightarrow \operatorname{Eq}\bigg(\prod_{i\in I} S^{-1}_{a_i} A \rightrightarrows \prod _{i,j\in I}  S^{-1}_{a_ia_j} A\bigg)\]
    is surjective.     
\end{sublemma}
\begin{proof}
    Suppose  we have an element in the equalizer, that is a collection of elements $x_i/s_i\in S^{-1}_{a_i} A$ satisfying $x_i/s_i=x_j/s_j$ in $S^{-1}_{a_i a_j} A$ for all pair of indices $i,j\in I$.
    We want to find an element $a\in A$ such that the equality $a/1_A=x_i/s_i$ holds in $S^{-1}_{a_i} A$ for all~$i\in I$. 
    
    As the open sets $D(a_i)$ cover $\Spec A$, by \cref{lem: principal open sets are basis for Zariski topology of Spec} we can choose a finite set $J\subseteq I$ such that $\bigcup_{j\in J} D(a_j)=\Spec A$. 
    
    Now, as $J$ is finite we can use \cref{PreviousLemma} to deduce the existence of $x_j'\in A$ and $s'_j\in S_{a_j}$ for every $j\in J$ such that \[\frac{x_j'}{s_j'}=\frac{x_j}{s_j}\in S^{-1}_{a_j} A\] and satisfying $x_j's_k'=x_k's_j'$ for all~$j,k\in J$.

    By definition, the fact that $s'_j\in S_{a_j}$ implies that $D(s_j^\prime)\supseteq D(a_j)$ and so \[\bigcup_{j\in J} D(s_j')=\Spec A\] since the family of $D(a_j)$ with $j\in J$ already covers~$\Spec A$. 
    In particular, because of \cref{lem: principal open sets are basis for Zariski topology of Spec} there exists $b_j\in A$ with $\sum_{j\in J} b_j s_j'=1_A$.

    We claim that the element $a= \sum_{j\in J} b_jx_j' \in A$ is the one we are seeking.
    Indeed, for all $j\in J$ we have 
    \[a s_j'= \sum_{k\in J} b_kx_k's_j' = \sum_{k\in J} b_kx_j's_k^\prime= x_j',\] and therefore 
    \begin{equation}\label{eq: equality in surjectivity global}
    \frac{a}{1_A}= \frac{x_j'}{s_j'}=\frac{x_j}{s_j}\in S^{-1}_{a_j} A.
    \end{equation}
    Instead, if $i\in I\setminus J$ consider the covering of $D(a_i)$ given by 
    \[D(a_i)= \bigcup_{j\in J} (D(a_i)\cap D(a_j))=\bigcup_{j\in J} D(a_ia_j),\]
    which holds since the family of $D(a_j)$ with $j\in J$ already covers~$\Spec A$.
    In order to see that $a/1_A=x_i/s_i$ in $S^{-1}_{a_i} A$ we can use \cref{subl: injectivity principal} to reduce it to the proof of $a/1_A=x_i/s_i$ in $S^{-1}_{a_ia_j} A$ for all~$j\in J$.
    But we have
    \[\frac{x_i}{s_i}=\frac{x_j}{s_j}=\frac{a}{1_A}\in S^{-1}_{a_ia_j} A\]
    because of \eqref{eq: equality in surjectivity global}. 
    \end{proof}

We are finally ready to conclude the proof of our main claim.

\begin{proof}[Proof of \cref{lem: SheafLemma}]
Injectivity follows immediately from \cref{subl: injectivity principal}.
Regarding surjectivity, we already know the result for the whole space $\Spec A$ by \cref{subl: surjectivity global}, and one deduces the general case from this one by identifying $D(a)$ with $\Spec S_a^{-1}A$ and using \cref{subl: localization of localization}, as done in the proof of \cref{subl: injectivity principal}.
\end{proof}

\subsection*{Approach C: Sheaf of sections}

Our last approach to the construction of the structure sheaf on $X$ readily follows Hartshorne's book \cite[section~II.2]{Hartshorne}, and it has the advantage that it directly defines the semiring of its local sections on any open set without depending on any choice. 

Let $U \subseteq X$ be an open subset.
We define $\calO_X(U)$ to be the set of functions
\[
s\colon U \longrightarrow \coprod_{\p \in U} A_{\p}
\]
such that $s(\mathfrak{p})\in A_\mathfrak{p}$ for each $\mathfrak{p} \in U$ and such that $s$ is a local quotient of elements of~$A$.
More precisely, the latter means that for each $\mathfrak{p} \in U$ there is a neighbourhood $V$ of~$\p$, contained in~$U$, and elements $a,f \in A$, such that for each $\mathfrak{q} \in V$ it happens that $f \notin \mathfrak{q}$ and $s(\mathfrak{q})=a/f$ in~$A_\mathfrak{q}$.

Sums and products of such functions are again such, and the function mapping each $\p$ to $1_{A_\p}$ is an identity.
Thus $\calO_X(U)$ is a commutative semiring with identity.
If $V \subseteq U$ are two open sets, the natural restriction map $\calO_X(U) \to \calO_X(V)$ is a homomorphism of semirings.
One shows from the local nature of the definition that $\calO_X$ is a sheaf, with stalk $A_\mathfrak{p}$ at each~$\mathfrak{p}$.

\subsection{Equivalence of definitions}

We show here that the three approaches presented in the previous subsection for the construction of the structure sheaf on $X$ are actually equivalent.

\begin{theorem}\label{thm: equivalence of definitions}
    The three definitions of $\calO_X$ from \cref{subsec: three constructions of the structure sheaf on Spec} produce the same sheaf on~$X$, up to isomorphism.
    In particular:
    \begin{enumerate}
        \item For any element~$a\in A$, the natural morphism $A_a\to \calO_X(D(a))$ is an isomorphism.
        \item The semiring of global sections of $\mathcal{O}_X$ is~$\calO_X(X) \cong A$.
        \item For any~$\mathfrak{p} \in X$, the stalk $\calO_{X,\p}$ of the sheaf $\calO_X$ at $\mathfrak{p}$ is isomorphic to the local semiring~$A_\p$.
    \end{enumerate}
\end{theorem}
\begin{proof}
    Let $\mathcal{F}$ be the presheaf of semirings on $X$ defined by
    \[
    \mathcal{F}(U)=S_U^{-1}A
    \]
    for all open subset~$U\subseteq X$, as in the first approach in \cref{subsec: three constructions of the structure sheaf on Spec}.
    The restriction of $\mathcal{F}$ to the basis of principal open sets of~$X$ has the same stalks as~$\mathcal{F}$ and, as we proved in \cref{lem: SheafLemma}, is a sheaf on such a basis.
    A comparison between the definition of the sheafification from \cite[Section~6.17]{stacks-project} and the proof of \cite[Lemma~6.30.6]{stacks-project} shows then that the first two approaches in \cref{subsec: three constructions of the structure sheaf on Spec} produce the same sheaf on~$X$.

    It remains to compare $\mathcal{F}^\#$ with the sheaf $\mathcal{O}_X$ defined through the third approach.
    To do so, consider the morphism of presheaves $\mathcal{F}\to\mathcal{O}_X$ defined on all open subset~$U\subseteq X$ by
    \begin{equation}\label{eq: morphism of presheaves}
        \mathcal{F}(U)\longrightarrow\mathcal{O}_X(U), \quad a/f\longmapsto(\mathfrak{p}\mapsto a/f)
    \end{equation}
    for all~$a/f\in S_U^{-1}A$.
    Since $f\in S_U$ it follows that $f\notin\mathfrak{p}$ for all $\mathfrak{p}\in U$ and so $a/f$ can be seen as an element in~$A_\mathfrak{p}$.
    It is easily seen that the morphism is well-defined.
    For all $\mathfrak{p}\in X$ the induced morphism between stalks is
    \[
    \mathcal{F}_\mathfrak{p}\longrightarrow\mathcal{O}_{X,\mathfrak{p}}, \quad [a/f]\longmapsto a/f
    \]
    where $a/f\in\mathcal{F}(U)$ for some open set $U$ containing~$\mathfrak{p}$.
    Notice that this is in fact an isomorphism between $\mathcal{F}_\p$ and~$\mathcal{O}_{X,\mathfrak{p}}=A_\p$.
    Indeed, for any~$a/f\in A_\p$, with~$f\notin\mathfrak{p}$, we can take~$U=D(f)\ni \mathfrak{p}$; the germ of the element $a/f\in \mathcal{F}(D(f))$ is a preimage of~$a/f$, thus proving surjectivity.
    For the injectivity, if~$a/f=b/g  \in A_\p$, then there exists $u\in A\setminus \p$ such that~$agu=bfu$; therefore $a/f=b/g$ in~$\mathcal{F}(D(fgu))$, showing that~$[a/f]=[b/g]$. 
        
    Since the morphism defined in \eqref{eq: morphism of presheaves} induces an isomorphism at all stalks, its sheafification realizes the desired isomorphism between $\mathcal{F}^\#$ and~$\mathcal{O}_X$.

    Now, claim (1) follows from the definition of the structure sheaf via the second approach and from \cref{lem: submonoid of open set}.
    Claim (2) is a special case of (1), obtained by taking~$a=1_A$.
    Claim (3) has been shown directly.
\end{proof}

\begin{example}\label{ex: sheaf on Spec N}
    Consider the semiring of natural numbers~$\mathbb{N}$ and denote by $\mathfrak{m}=\mathbb{N}\setminus\{1\}$ its unique maximal ideal.
    It follows from the explict description in \cref{ex: Sp and Spec of N} that a nonempty open subset of $\Spec \mathbb{N}$ is either of the form $D(N)$ for some nonzero natural number~$N$, or it agrees with~$(\Spec\mathbb{N})\setminus\{\mathfrak{m}\}$.
    We get from \cref{thm: equivalence of definitions} that the local sections of the structure sheaf are
    \[
    \calO_{\Spec \mathbb{N}}(D(N))=\mathbb{N}[1/N]
    \]
    on the first type of open.
    For the second, writing $(\Spec\mathbb{N})\setminus \{\mathfrak{m}\}=D(2)\cup D(3)$ we deduce from \eqref{eq: sheaf from Approach B} that
    \[
    \calO_{\Spec \mathbb{N}}\left((\Spec\mathbb{N})\setminus \{\mathfrak{m}\}\right)=\mathbb{N}.
    \]
    
    Stalks are also immediately computed from \cref{thm: equivalence of definitions} and in particular we have that
    \[
    \calO_{\Spec \mathbb{N},\{0\}}\cong\mathbb{Q}_{\geq0}\quad\text{and}\quad\calO_{\Spec \mathbb{N},\mathfrak{m}}\cong\mathbb{N}.
    \]
\end{example}


\subsection{The category of affine semischemes}

Recall that in classical algebraic geometry the functor $\Spec$ induces an anti-equivalence between the category of commutative rings and the category of affine schemes. We extend this correspondence to the setting of commutative semirings by replacing rings with semirings and locally ringed spaces with the following natural analogue.

\begin{definition}
A \emph{locally  semiringed space} is a pair $(\mathfrak{X}, \mathcal{O}_\mathfrak{X})$ consisting of a topological space $\mathfrak{X}$ and a sheaf of commutative semirings $\mathcal{O}_\mathfrak{X}$ on $\mathfrak{X}$ such that for all $x \in \mathfrak{X}$ the stalk $\calO_{\mathfrak{X},x}$ is a local semiring.
A \emph{morphism of locally semiringed spaces}
\[
(f, f^\#)\colon (\mathfrak{X},\mathcal{O}_\mathfrak{X})\longrightarrow (\mathfrak{Y},\mathcal{O}_\mathfrak{Y})
\]
is a continuous map $f\colon \mathfrak{X}\to \mathfrak{Y}$ together with a morphism of sheaves of semirings
\[
f^\#\colon \mathcal{O}_\mathfrak{Y} \longrightarrow f_*\mathcal{O}_\mathfrak{X}
\]
such that for all $x\in \mathfrak{X}$ the induced homomorphism on stalks
\[
f^\#_x\colon (f^{-1}\mathcal{O}_\mathfrak{Y})_x = \mathcal{O}_{\mathfrak{Y},f(x)} \longrightarrow
\mathcal{O}_{\mathfrak{X},x}
\]
is a local homomorphism of semirings, that is it sends the maximal ideal of
$\mathcal{O}_{\mathfrak{Y},f(x)}$ into the maximal ideal of~$\mathcal{O}_{\mathfrak{X},x}$. 

A (locally) semiringed space $(\mathfrak{X},\mathcal{O}_\mathfrak{X})$ is called an \emph{affine semischeme} if it is isomorphic to one of the form $(\Spec A, \mathcal{O}_{\Spec A})$ for some commutative semiring~$A$.  
\end{definition}

The following is a categorical version of \cite[Proposition~II.2.3]{Hartshorne} for the case of affine locally semiringed spaces.

\begin{proposition}\label{prop: category Spec is equivalent}
The contravariant functor
\[
\Spec \colon A\longmapsto (\Spec A, \mathcal{O}_{\Spec A})
\]
is an anti-equivalence between the category of commutative semirings and the category of affine semischemes. 
\end{proposition}
\begin{proof}
The proof works as in the classical case of rings, by showing that the functor $\Spec$ is fully faithful and essentially surjective, with quasi-inverse the global section functor.

First, for any pair of semirings $A,B$ the functor $\Spec$ gives a map
\begin{equation}\label{eq: functor Spec on morphism}
\Spec\colon\Hom(A,B) \longrightarrow \Hom((\Spec B, \mathcal{O}_{\Spec B}),(\Spec A, \mathcal{O}_{\Spec A}))
\end{equation}
obtained by sending $\varphi$ to the pair $(\Spec(\varphi),\Spec(\varphi)^\#)$ where $\Spec(\varphi)$ is the continuous map from \cref{lem: continuity of morphisms}, while $\Spec(\varphi)^\#$ is defined by globalizing the localizations of~$\varphi$ through the third approach to the structure sheaf proposed in \cref{subsec: three constructions of the structure sheaf on Spec}, see \cite[Proposition~II.2.3]{Hartshorne} for the analogue construction in the case of rings.
Notice this is a morphism of locally semiringed spaces since, for any prime ideal $\p\in \Spec B$, the induced morphism $A_{\varphi^{-1}(\p)} \to B_{\p}$ is a local morphism of local semirings. 

The fact that $\Spec$ is essentially surjective is obvious from the definition of affine semischemes.
To show that it is fully faithful, 
observe that, if we have a morphism $(f,f^\#)\colon (\Spec B, \mathcal{O}_{\Spec B}) \to (\Spec A, \mathcal{O}_{\Spec A})$
as locally semiringed spaces, then in particular we have a semiring morphism 
\[
    f^\#(\Spec A)\colon \mathcal{O}_{\Spec A}(\Spec A)\to  f_*\mathcal{O}_{\Spec B}(\Spec A)=\mathcal{O}_{\Spec B}(\Spec B).
\]
Identifying the global sections with the corresponding semiring via the isomorphism given in \cref{thm: equivalence of definitions}, one obtains a morphism~$\varphi\colon A\to B$.
This gives an inverse of the map in \eqref{eq: functor Spec on morphism}.
For example, if we take $\p\in \Spec A$, using the commutative diagram
\begin{center}
    \begin{tikzcd}[column sep=small]
    A\cong\mathcal{O}_{\Spec A}(\Spec A) \arrow{r}  \arrow{d}{f^\#} 
    & \mathcal{O}_{\Spec A,\p}\cong A_{\p} \arrow{d}{} \\
   B\cong\mathcal{O}_{\Spec B}(\Spec B) \arrow{r}  & \mathcal{O}_{\Spec B,f(\p)}\cong B_{f(\p)}
\end{tikzcd}
\end{center}
and since the right vertical arrow is local, we see that the maximal ideal in $\mathcal{O}_{\Spec B,f(\p)}$ is the inverse image of the maximal ideal of $\mathcal{O}_{\Spec A,\p}$, and therefore $\Spec(\varphi)(\p)=(f^\#)^{-1}(\p)=f(\p)$. This shows that the continuous map $\Spec(\varphi)$ is the original map~$f$.
In order to show the associated morphism of sheaves is also the same, one can do it just for the basis of principal open sets. In this case, first one shows that for every $a\in A$, $f^{-1}(D(a))=D(\varphi(a))$, and then uses the commutative diagram 
\begin{center}
    \begin{tikzcd}[column sep=small]
    A\cong\mathcal{O}_{\Spec A}(\Spec A) \arrow{r}  \arrow{d}{f^\#} 
    & \mathcal{O}_{\Spec A}(D(a))=S_a^{-1}A \arrow{d}{} \\
   B\cong\mathcal{O}_{\Spec B}(\Spec B) \arrow{r}  & \mathcal{O}_{\Spec B}(D(\varphi(a)))=S_{\varphi(a)}^{-1}B
\end{tikzcd}
\end{center}
to show that the right vertical morphism is the one induced by the universal property of the localization. 
\end{proof}

This construction of the category of affine semischemes should be the same one obtains using Toën and Vaquié general theory \cite{ToenVaquie2009}; this can be shown by similar arguments like in \cite{Marty2012}.

\subsection{The category of semischemes}\label{sec: semiring schemes}

The category of semischemes can then be defined as the full subcategory of the category of locally semiringed spaces which are locally affine. 

\begin{definition}\label{def: semiring schemes}
    A \emph{semischeme} is a (locally) semiringed space $(\mathfrak{X}, \mathcal{O}_\mathfrak{X})$ such that, for every $x\in \mathfrak{X}$, there exist an open neighbourhood $U\subseteq\mathfrak{X}$ of~$x$, a semiring $A_x$ and an isomorphism as (locally) semiringed space between $(U,(\mathcal{O}_\mathfrak{X})_{|U})$ and $ (\Spec A_x, \mathcal{O}_{\Spec A_x})$.
\end{definition}

Observe that if a semiringed space is isomorphic to a locally semiringed space, then it is a locally semiringed space. 

A first example of a semischeme that is not affine in general is given by open subsemischemes of an affine semischeme.
Given any semischeme $(\mathfrak{X},\mathcal{O}_\mathfrak{X})$ and an open $U\subseteq\mathfrak{X}$, the associated \emph{open subsemischeme} is just by definition $(U,(\mathcal{O}_\mathfrak{X})_{|U})$.
Observe that $(U,(\mathcal{O}_\mathfrak{X})_{|U})$ is a semischeme according to \cref{def: semiring schemes}, since $U$ is locally contained in an affine semischeme, and then any of its open subsets can be covered by affine open sets, for example some principal open sets. 

The following example of a non-affine open subsemischeme is analogue to the well-known construction in the case of rings.  

\begin{example}
Let $F$ be a semifield, and let $A=F[x,y]$ be the semiring of polynomials in two variables with coefficients in~$F$.
The ideal $\mathfrak{m}=xA+yA\subseteq A$ is maximal and so it corresponds to a closed point in~$\Spec A$.
Then, its complement $U=(\Spec A)\setminus\{\mathfrak{m}\}$ is open.
We claim that the open subsemischeme $(U,(\mathcal{O}_{\Spec A})_{|U})$ is not affine.

Indeed, taking the covering by principal open sets $U=D(x)\cup D(y)$ and using \eqref{eq: sheaf from Approach B} and \cref{thm: equivalence of definitions}, we obtain that 
\[
(\mathcal{O}_{\Spec A})_{|U}(U)
=
\mathcal{O}_{\Spec A}(U)
\cong
\operatorname{Eq} (A_x\times A_y \rightrightarrows A_{xy})
\cong A.
\]
Therefore, if $(U,(\mathcal{O}_{\Spec A})_{|U})$ was affine, then by \cref{prop: category Spec is equivalent} it would be isomorphic to~$\Spec A$, giving a contradiction.
\end{example}

While open subsets of $\Spec A$ naturally inherit a semischeme structure, for a closed subset~$V(I)\subseteq \Spec A$, with $I$ an ideal of~$A$, it is not always possible to define a natural structure of a semischeme on it, unless $I$ is a kernel. 

\begin{example}
    Consider the maximal ideal $\mathfrak{m}=\mathbb{N}\setminus \{1\}$ of $\mathbb{N}$ from \cref{ex: a prime ideal of N}.
    We will now show that the closed set $Z=\{\mathfrak{m}\}=V(\mathfrak{m}) \subseteq \Spec \mathbb{N}$ has no structure of an affine subsemischeme.

    To do so, consider a morphism of affine semischemes
    \[(f,f^\#)\colon (\Spec A,\mathcal{O}_{\Spec A})\to(\Spec \mathbb{N},\mathcal{O}_{\Spec \mathbb{N}})\]
    such that~$f(\Spec A)\subseteq Z=\{\mathfrak{m}\}$, and let $\varphi\colon \mathbb{N} \to A$ be the corresponding morphism between global sections given by \cref{prop: category Spec is equivalent}.
    If~$A\ne\{0_A\}$, there would exist a prime kernel $\p$ in~$\Spec A$; but then $\varphi^{-1}(\p)=f(\p)=\mathfrak{m}$, which can not be true as $\varphi^{-1}(\p)$ is a kernel, by \cref{lem: preimage of subtractive ideal}, while $\mathfrak{m}$ is not, because of \cref{rem: relation between kernel and quotient}.
    Therefore $A=\{0_A\}$, and so $\Spec A$ is the empty semischeme.
    
    It follows that $Z$ is not the image of an affine semischeme, and so it cannot be given the structure of an affine subsemischeme of~$\Spec \mathbb{N}$.
    A similar argument shows that $Z$ cannot be a non-affine subsemischeme either.  
\end{example}

Other well-known constructions in the ring case can be carried out analogously for semirings. For example, one can define the \emph{projective $n$-space} over a semifield as a semischeme, by gluing $n$-dimensional affine spaces along their intersections, as follows. 

\begin{example}
Let $F$ be a semifield, and consider the polynomial semiring $F[T_0,\dots,T_n]$ graded by total degree.  
For each $i=0,\dots,n$, we form the degree‑zero part of the localization $F[T_0,\dots,T_n]_{T_i}$ and denote it by
\[
A_i
=
F\!\left[\frac{T_0}{T_i},\ldots,
\widehat{\frac{T_i}{T_i}},
\ldots,\frac{T_n}{T_i}\right].
\]
We denote the associated affine semischeme by~$U_i = \Spec A_i$.
For~$i\neq j$, consider the open subset $V_{ij}\subseteq U_i$ described by localizing further at~$T_j/T_i$.
Intuitively, $V_{ij}$ agrees with the intersection~$U_i\cap U_j$, seen as an open subset of~$U_i$, and the element $T_j/T_i$ is invertible in this intersection. 
Thus
\[
V_{ij}=\Spec A_{ij},
\qquad
A_{ij}
= A_i\!\left[ \left(\frac{T_j}{T_i}\right)^{-1} \right].
\]
Similarly, $V_{ji}=\Spec A_{ji}$ with
\[
A_{ji}
= A_j\!\left[ \left(\frac{T_i}{T_j}\right)^{-1} \right].
\]
The canonical identifications $A_{ij}\cong A_{ji}$ induce isomorphisms of affine semischemes
\[
V_{ij} \cong V_{ji}.
\]
These gluing isomorphisms satisfy the usual cocycle condition on triple intersections, and therefore the $n+1$ affine semischemes $U_0,\dots,U_n$ glue uniquely to a semischeme~$\mathbb{P}^n_F.$

\end{example}

In the proof of \cref{prop: category Spec is equivalent} we showed that the morphisms as locally semiringed spaces between two affine semischemes are the same as the morphisms of the associated semirings.
This generalizes to the following property, which can be proved as for rings, see \cite[Lemma~26.6.4]{stacks-project}.

\begin{proposition}\label{prop: Morphism to an affine scheme}
    Let $(\mathfrak{X},\mathcal{O}_\mathfrak{X})$ be any locally semiringed space, and let $A$ be a semiring. Then the map 
    \[\Hom((\mathfrak{X},\mathcal{O}_\mathfrak{X}), (\Spec A,\mathcal{O}_{\Spec A})) \longrightarrow \Hom (A, \mathcal{O}_\mathfrak{X}(\mathfrak{X}))\]
    obtained by sending $(f,f^\#)$ to $f^\#(\Spec A)$ is a bijection.
\end{proposition}

\section{The structure sheaf for \texorpdfstring{$\Sp$}{}}\label{sec: structure sheaf for Sp}

We fix again for the entire section a semiring $A$ and we now let $X = \Sp A$ be the associated $k$-spectrum, seen as a topological space like in \cref{sec: spectrum constructions}.

We aim to emulate the constructions of \cref{sec: structure sheaf for Spec} to obtain a suitable structure sheaf on~$X$.
In the process, we highlight some new difficulties in the determination of its local sections on principal open subsets, motivating the definition of global semirings in \cref{subs: Global}.

\subsection{The construction of the structure sheaf}

Analogously to the definition in \eqref{eq: multiplicative submonoid associate to open U}, for any open subset $U\subseteq X$ we set
\begin{equation}
    \widetilde{S}_U = \{b\in A\mid \bD(b)\supseteq U \}\subseteq A.
\end{equation}
In the particular case in which $U=\bD(a)$ for some $a\in A$ we write for short
\[
\widetilde{S}_a=\widetilde{S}_{\bD(a)}.
\]

The following statement concerns the algebraic properties of such sets.

\begin{lemma}\label{lem: submonoid of open set for kernels}
The set $\widetilde{S}_{1_A}$ is the set of semi-invertible elements of~$A$.
Moreover, for any open subset~$U\subseteq X$, the set $\widetilde{S}_U$ is a saturated multiplicative submonoid of~$A$ and the localization $\widetilde{S}_U^{-1}A$ is a hard semiring.
\end{lemma}
\begin{proof}
    The first claim is a direct consequence of \cref{lem: principal open sets are basis for Zariski topology of Sp}.
    The first part of the second is proved as \cref{lem: submonoid of open set}, by using \cref{lem: principal open sets are basis for Zariski topology of Sp} instead of \cref{lem: principal open sets are basis for Zariski topology of Spec}.


    To prove the last claim, we follow the same argument as in \cref{prop: hardening is hard} and \cref{lem: localization at prime kernel is hard}. It is enough to show that any semi-invertible element $a/s$ of~$\widetilde{S}_U^{-1}A$ satisfies that~$a\in \widetilde{S}_U$, or equivalently that~$\bD(a)\supseteq U$. 

Since $a/s$ is semi-invertible, there exist $b,c\in A$ and $u,v,w\in\widetilde{S}_U$, such that 
$s^2uvw + a (bsvw)=a(csuw)$. 
Hence, $s^2uvw$ is in the subtractive ideal generated by $a$, and therefore $\bD(a)\supseteq \bD(s^2uvw)\supseteq U$ by \cref{lem: principal open sets are basis for Zariski topology of Sp}. 
\end{proof}

\begin{remark}\label{rem: S1 and S1tilde}
For all open set $U\subseteq\Spec A$ it is immediate to verify that
\[
S_U\subseteq\widetilde{S}_{U\cap X}.
\]
In particular, $S_a\subseteq\widetilde{S}_a$ for all~$a\in A$.
Notice that the inclusion can be strict: for example, it follows from \cref{lem: principal open sets are basis for Zariski topology of Spec,lem: submonoid of open set for kernels} that $S_{1_A}$ is the set of invertible elements in~$A$, while $\widetilde{S}_{1_A}$ consists of its semi-invertible elements.
Recall that these two sets differ for non-hard semirings such as for~$A=\mathbb{B}[x]$, see \cref{ex: non-hard semiring}.
\end{remark}

We can now follow the same approaches as in \cref{subsec: three constructions of the structure sheaf on Spec} to construct the structure sheaf on~$X$.

The first option is to consider the presheaf of semirings $\mathcal{L}_A$ on $X$ given by
\begin{equation}\label{eq: presheaf LA for Sp A}
    \mathcal{L}_A(U)=\widetilde{S}_U^{-1} A
\end{equation}for all open subset~$U\subseteq X$, and to define the structure sheaf
\[
\mathcal{O}_X= \mathcal{L}_A^\#
\]
on $X$ to be its sheafification.

\begin{example}
    It could happen that the localization presheaf is already a sheaf.
    For instance, consider the semiring of natural numbers~$\mathbb{N}$ and recall that, as seen in \cref{ex: Sp and Spec of N}, we have a natural identification
    \[
    \Sp \mathbb{N}\cong\Spec \mathbb{Z}.
    \]
    The localization presheaf $\mathcal{L}_{\mathbb{N}}$ on $\Sp\mathbb{N}$ is already a sheaf.
    In fact, similarly to \cref{ex: sheaf on Spec N}, any nonempty open subset of $\Sp \mathbb{N}$ is of the form $\bD(N)$ for some nonzero natural number~$N$, and it is easily verified that
    \[
    \mathcal{L_\mathbb{N}}(\bD(N))=\mathbb{N}[1/N]=\calO_{\Spec \mathbb{Z}}(D(N))\cap \mathbb{Q}_{\ge 0}.
    \]
    
    The semifield of positive rational numbers $\mathbb{Q}_{\ge 0}$ is the semifield of fractions of~$\mathbb{N}$, and also the stalk at $0$ of the sheaf~$\mathcal{O}_{\Sp\mathbb{N}}=\mathcal{L}_{\mathbb{N}}$.
\end{example}

Alternatively, one can replicate the approach C from \cref{subsec: three constructions of the structure sheaf on Spec} to construct the structure sheaf $\calO_X$ directly, following \cite[section~II.2]{Hartshorne} or \cite[section~6.17]{stacks-project}. 

As in \cref{thm: equivalence of definitions}, these two approaches give isomorphic sheaves on~$X$.

\begin{remark}\label{rem: A and Ahard have isomorphic semiringed space}
    Thanks to \cref{lem: submonoid of open set for kernels}, the semiring $\mathcal{L}_A(U)$ is always hard, and in fact it is shown in \cite{MasdeuRoeXarles} that, if $\chi\colon A\to A^\diamondsuit$ denotes the hardening morphism, then $\Sp(\chi)_*\mathcal{L}_{A^\diamondsuit}\cong \mathcal{L}_A$.
    It follows that the homeomorphism $\Sp(\chi)$ from \cref{lem: hardering on Sp} induces an isomorphism of semiringed spaces
    \[
    (\Sp A,\mathcal{O}_{\Sp A})\cong(\Sp A^\diamondsuit,\mathcal{O}_{\Sp A^\diamondsuit}).
    \]
    As a result, when studying the $k$-spectral version of affine semiring schemes, it is enough to restrict to the case of hard semirings.
\end{remark}

It would be advantageous for computations to have a basis for the topology of $X$ over which the local sections of $\mathcal{O}_X$ are known.
In the spectral case, this was ensured by the fact that the restriction of the localization presheaf to principal open subsets is a sheaf on such a basis, as seen in the second approach in \cref{subsec: three constructions of the structure sheaf on Spec}.

To make this work in the $k$-spectral setting as well, we should check whether the presheaf $\mathcal{L}_A$ is a sheaf on the basis of principal open sets of~$X$.
More precisely, and analogously to \cref{lem: SheafLemma}, we would need to show that, given $a\in A$ and  $a_i \in A$ for $i\in I$ such that $\bigcup_{i\in I} \bD(a_i) = \bD(a)$, the morphism
\begin{equation}\label{eq: exactness on principal basis for Sp structure sheaf}
    \widetilde{S}_a^{-1} A \longrightarrow\operatorname{Eq}\bigg(\prod_{i\in I} \widetilde{S}_{a_i}^{-1} A \rightrightarrows \prod _{i,j\in I}  \widetilde{S}_{a_ia_j}^{-1} A\bigg)
\end{equation}
is an isomorphism.

\begin{remark}\label{rem: exactness for Sp sheaf in the global case}
    It is already interesting to study what happens in the global case, that is when~$a=1_A$.
    In such a situation, \cref{lem: submonoid of open set for kernels} and \eqref{eq: definition of hardening} imply
    \[
    \widetilde{S}_{1_A}^{-1}A=A^\diamondsuit.
    \]
    Moreover, by \cref{lem: principal open sets are basis for Zariski topology of Sp} a collection of principal open sets $\bD(a_i)$ is a covering of $\Sp A$ if and only if the subtractive ideal generated by the set of $a_i$ is equal to~$A$.
    
    Hence, checking \eqref{eq: exactness on principal basis for Sp structure sheaf} when $a=1_A$ translates into showing that for any family of $a_i$ with~$i\in I$ such that the subtractive ideal generated by $a_i$ is~$A$, the morphism
    \begin{equation}\label{eq: global exactness on principal basis for Sp structure sheaf}
    A^\diamondsuit\longrightarrow\operatorname{Eq}\bigg(\prod_{i\in I} \widetilde{S}_{a_i}^{-1} A \rightrightarrows \prod _{i,j\in I}  \widetilde{S}_{a_ia_j}^{-1} A\bigg)
    \end{equation}
    is an isomorphism.
\end{remark}

Unfortunately, it turns out that the morphism in \eqref{eq: exactness on principal basis for Sp structure sheaf} is not an isomorphism in general.
In fact, it can fail to be injective even in the global case considered in~\cref{rem: exactness for Sp sheaf in the global case}.

\begin{example}[\cite{MasdeuRoeXarles}]
    Let $A= \mathbb{N}[x,y]/(x^2\sim x, y^2 \sim y, 1+x \sim x+y)$.
    This semiring is hard, as its only semi-invertible element is the class of~$1$.
    Notice that~$\overline{xA+yA}=A$, so that~$\Sp A=\bD(x)\cup\bD(y)$.
    One can show that $1+xy$ and $x+y$ have the same image via $ A \rightarrow A[x^{-1}]\times A[y^{-1}]$, namely $\left((1+y)/1,(1+x)/1\right)$, but $1+xy\neq x+y$ in~$A$, contradicting the injectivity of~\eqref{eq: global exactness on principal basis for Sp structure sheaf}.
\end{example}

However, the injectivity of the morphism in \eqref{eq: global exactness on principal basis for Sp structure sheaf} is guaranteed in the idempotent case, as the next result shows.

\begin{lemma}
If $A$ is a hard idempotent semiring, and $a_i\in A$ for $i\in I$ are such that $\Sp A=\bigcup_{i\in I} \bD(a_i)$, then the morphism
\[A^\diamondsuit\longrightarrow \prod_{i\in I} \widetilde{S}_{a_i}^{-1} A \] is injective.    
\end{lemma}

\begin{proof}
As the semiring is hard, we have that~$A^\diamondsuit\cong A$.
The key point is that, for a hard idempotent semiring, an ideal whose subtractive closure is the total is indeed the total ideal by \cref{prop: ideals whose kernel is total}.
With this observation, the desired injectivity is proved similarly as in \cref{subl: injectivity global}.
\end{proof}

As for the surjectivity of the morphism in \eqref{eq: global exactness on principal basis for Sp structure sheaf}, we do not have a counterexample, but we suspect it is not true in general.
Still, it trivially holds in some important elementary cases, as in the one of the next example. 

\begin{example}\label{ex: trivial coverings on Polynomials}
    Let $F$ be an idempotent semifield, and consider the semiring~$A=F[x_1,\ldots,x_n]$.
    Then, the morphism \eqref{eq: global exactness on principal basis for Sp structure sheaf} for such $A$ is an isomorphism.  
    This follows from the stronger fact that any covering $\Sp A=\bigcup_{i\in I} \bD(a_i)$ is trivial, in the sense that $\Sp A=\bD(a_i)$ for some~$i\in I$. 

    To show this last claim, consider the maximal ideal~$\mathfrak{m}=x_1A+\ldots+x_nA$.
    Since  $\Sp A=\bigcup_{i\in I} \bD(a_i)$, there exists $i\in I$ such that~$\mathfrak{m}\in\bD(a_i)$.
    This means that $a_i\notin\mathfrak{m}$ and therefore $a_i$ is a polynomial with non-zero free term.
    By \cref{lem: semi-invertibles in polynomial ring} it is semi-invertible, which implies that $\bD(a_i)=\Sp A$ by \cref{lem: principal open sets are basis for Zariski topology of Sp}.
    
    Notice that the same argument applies to any semiring which is kernel-local, meaning it has a unique maximal kernel.
\end{example}

\subsection{Global semirings}\label{subs: Global}

Recall that the semiring of global sections of the structure sheaf on $\Spec A$ is isomorphic to~$A$, as seen in \cref{thm: equivalence of definitions}.

This is no longer true in the $k$-spectral case, as the natural morphism
\[
A\longrightarrow \Gamma A=\calO_{\Sp A}(\Sp A)
\] 
can fail to be an isomorphism.
It is then interesting to consider the map that associates to each semiring $A$ the corresponding semiring $\Gamma A$ of global sections of the structure sheaf. 

\begin{lemma}\label{lem: Gamma is functor}
    The mapping $A\mapsto \Gamma A$ is a functor from the category of semirings to itself.   
\end{lemma}
\begin{proof}
For any morphism $f\colon A\to B$, we can construct the corresponding morphism $\Gamma f\colon\Gamma A\to \Gamma B$ as follows.
    Notice that for all $a\in A$ we have
    \[
    \bD(f(a))=(\Sp(f))^{-1}(\bD(a)).
    \]
    Using this equality, one deduces that for any open subset $U\subseteq \Sp A$,
    \[
    f(\widetilde{S}_U)\subseteq \widetilde{S}_{U'},
    \]
    with~$U'=(\Sp(f))^{-1}(U)$.
    Therefore, the morphism $\varphi\circ f\colon A\to \widetilde{S}_{U'}^{-1}B$, with $\varphi\colon B\to \widetilde{S}_{U'}^{-1}B$ being the localization morphism, factors through a unique morphism 
    $\mathcal{L}_A(U)=\widetilde{S}_{U}^{-1}A \to \widetilde{S}_{U'}^{-1}B=\mathcal{L}_B(U')$.
    Standard properties of the sheafification imply that we have a morphism $\calO_{\Sp A}(U) \to \calO_{\Sp B}(U')$.
    By taking $U=\Sp A$ we get the desired morphism.
\end{proof}

We reserve a special name to the semirings for which the global sections of the $k$-spectral version of the structure sheaf agrees with the semiring itself.

\begin{definition}
    We say that a semiring $A$ is \emph{global} if the natural morphism $A\to\Gamma A $ is an isomorphism. 
\end{definition}

It is possible to show that any global semiring is hard (see \cite{MasdeuRoeXarles}). 
Conversely, \cref{rem: A and Ahard have isomorphic semiringed space} and \cref{ex: trivial coverings on Polynomials} show that the hardening of the polynomial ring in $n\ge 0$ variables over an idempotent semifield is global.
More generally, every kernel-local semiring $A$ with maximal kernel $\mathfrak{m}$ satisfies
\[
\Gamma A^\diamondsuit\cong \Gamma A=\calO_{\Sp A}(\Sp A)\cong A_{\mathfrak{m}}\cong A^{\diamondsuit},
\]
and so $A^\diamondsuit$ is a global semiring.


The following result gives a construction of a globalization functor, assigning to any semiring a global semiring in a universal way.  The proof is straightforward and left to the reader.  

\begin{proposition}
    Let $A$ be a semiring and consider the direct system $\{\Gamma^n A\mid n\ge 0\}$ with the natural morphisms $\Gamma^n A\to \Gamma^{n+1}A$.
    Then the direct colimit
    \[
    G(A)=\operatorname{colim}_{n} \Gamma^n A
    \]
    is a global semiring, and it comes with a morphism $A \to G(A)$ that is universal with respect to morphisms to the full subcategory of global semirings.
    
    Moreover, the mapping $G$ gives a functor to global semirings, left adjoint to the inclusion functor. 
\end{proposition}

Notice that for a global ring~$A$, the presheaf $\mathcal{L}_A$ in \eqref{eq: presheaf LA for Sp A} is not in principle a sheaf on the basis of principal open sets in~$\Sp A$.
Rather, the definition only implies that it verifies the sheaf property for the coverings \textit{of the total space}, as in the special case of \cref{rem: exactness for Sp sheaf in the global case}.

\section{Universal valuations}\label{sec:universalGval}

We now study a generalization of the definition of valuations to the context of semirings, taken from \cite[2.5]{Giansiracusa2X_2016}.
For further reading, the moduli space of such objects is studied in \cite{Giansiracusa2x2022} and \cite{Macpherson2020}.

We will see in \cref{thm: G-valuations induce homeomorphism sp spec} that these general valuations give a link between the spectrum of a semiring and the $k$-spectrum of a certain idempotent semiring associated to it, giving a distinguished sheaf on the latter.

\subsection{$G$-valuations on semirings}\label{subs: G-valuation}

Recall that any idempotent semiring $S$ is a partially ordered set with respect to the order defined by $s\preceq t$ whenever $s+t=t$.

\begin{definition}\label{def: G-valuation}
A \emph{$G$-valuation} on a semiring $A$ with values in an idempotent semiring $S$ is a map $v\colon A\rightarrow S$ satisfying:
\begin{itemize}
    \item[(1)] $v(0_A)=0_S$ and $v(1_A)=1_S$;
    \item[(2)] $v(ab)=v(a)v(b)$;
    \item[(3)] $v(a+b)\preceq v(a)+v(b)$.
\end{itemize}
We will denote by $\mathrm{Val}(A,S)$ the set of $G$-valuations in $A$ with values in~$S$.
\end{definition}

\begin{remark}
    With the above conventions, a $G$-valuation on a field $K$ with values in the tropical semifield $\mathbb{T}$ is nothing else than a usual non-Archimedean valuation on~$K$.
    Indeed, let $v\colon K \to \mathbb{T}$ be a~$G$-valuation.
    Then the third axiom reads $v(a+b)\preceq\min(v(a),v(b))$ and is equivalent to $$\min(v(a+b),v(a),v(b)) = \min(v(a),v(b))$$ and thus to $v(a+b) \geq \min(v(a),v(b))$ which yields the usual property of valuations. 
    
\end{remark}

Note that the kernel $\ker(v)=v^{-1}(0_S)$ of a $G$-valuation is an ideal of~$A$, but it is not in general a $k$-ideal, as $v$ is not necessarily a semiring homomorphism.
For example, the map $v\colon\mathbb{N}\to\mathbb{B}$ which sends all non-one natural numbers  to $0$ is a $G$-valuation, and its kernel is $\mathbb{N}\setminus\{1\}$, which is not a subtractive ideal by \cref{rem: relation between kernel and quotient}.

Given a $G$-valuation~$v\colon A\rightarrow S$, we set
\[
A_v=\{a\in A\mid v(a)\preceq 1_S\}.
\]
We see that $A_v$ is a subsemiring of~$A$.
It generalizes the notion of the valuation ring associated to a non-Archimedean valuation.

\begin{remark}\label{rem: G-valuation and Booleans}
    The $G$-valuations into the Boolean semiring $v\colon A\rightarrow \mathbb{B}$ are in a bijective correspondence with the prime ideals of~$A$.
    Given  $v\in \mathrm{Val}(A,\mathbb{B})$, its kernel is a prime ideal and thus $\ker(v)\in \mathrm{Spec} A$.
    The inverse assignment is obtained by mapping $\mathfrak{p}\in  \Spec A$ to the characteristic function $\chi_\mathfrak{p}$ defined by
    \[
    \chi_\mathfrak{p}(a) = \begin{cases}
    0 & \text{if } a \in \mathfrak{p},\\[4pt]
    1 & \text{if } a \notin \mathfrak{p}.
    \end{cases}
    \]

    Recall from \cref{lem: Sp of idempotent semiring} that if $A$ is idempotent, the same map restricts to a bijection between $\Sp A$ and $\operatorname{Hom} (A,\mathbb{B})\subseteq\mathrm{Val}(A,\mathbb{B})$.
\end{remark}

The following statement relates $G$-valuations with spectra.

\begin{proposition}\label{prop: valuation induces morphism of spectra}
Let $A$ be a semiring, $S$ an idempotent semiring, and $v\colon A\to S$ a $G$-valuation.
Then $v$ induces a continuous map
\[
v^*\colon \Sp S\longrightarrow \Spec A,
\qquad 
\mathfrak{p}\longmapsto v^{-1}(\mathfrak{p}).
\]
\end{proposition}

\begin{proof}
Under the identifications of \cref{rem: G-valuation and Booleans} and \cref{lem: Sp of idempotent semiring}, $v^*$ corresponds to the map $\mathrm{Hom} (S,\mathbb{B}) \rightarrow \mathrm{Val} (A,\mathbb{B})$ given by~$f\mapsto f\circ v$.
Hence, it is easily checked that $v^*$ is well-defined.

Moreover, for any $a\in A$ we have
\[
\begin{aligned}
(v^*)^{-1}(D(a))
&=\{\mathfrak{p}\in\Sp S\mid v^*(\mathfrak{p})\in D(a)\}
=\{\mathfrak{p}\in\Sp S\mid a\notin v^{-1}(\mathfrak{p})\}\\
&=\{\mathfrak{p}\in\Sp S\mid v(a)\notin\mathfrak{p}\}
=\bD(v(a)).
\end{aligned}
\]
Therefore, the inverse image of a principal open set in $\Spec A$ is open in~$\Sp S$.
As the former are a basis for the Zariski topology of $\Spec A$ by \cref{lem: principal open sets are basis for Zariski topology of Spec}, the map $v^*$ is continuous. 
\end{proof}

\subsection{The universal $G$-valuation}\label{subs: universal G-valuation}

Let us fix for this subsection the choice of a semiring $R$ and of an \emph{$R$-semialgebra}~$A$, that is a semiring equipped with a homomorphism~$\iota\colon R\rightarrow A$.
For example, any semiring has a unique structure of a $\mathbb{N}$-algebra, and any idempotent semiring that of a $\mathbb{B}$-algebra. 

We consider the set
\[M_R(A)\]
of finitely generated $R$-subsemimodules of~$A$.
It is an idempotent semiring with $M+N$ being the sum as $R$-semimodules and  $M\cdot N=\langle ab \mid a\in M , b\in N \rangle $ for any~$M,N\in M_R(A)$.
The identity with respect to addition is $\langle0_A\rangle=\{0_A\}$ while the one with respect to multiplication is~$\langle 1_A\rangle_R=\iota(R)$.

\begin{example}
    When~$A=R$, we get the idempotent semiring of finitely generated ideals of~$A$, with the usual sum and product of ideals; it was studied already by Dedekind in the case of commutative rings, see \cite[Example 1.4]{Golan1999}.
    The module version of this construction (for commutative rings) was also considered in \cite{BarilBoudreau_Garay}.
\end{example}

Note that the partial order on the idempotent semiring $M_R(A)$ is just~$\subseteq$, as for two submodules $M,N\in M_R(A)$ we have that $M+N=N$ if and only if~$M\subseteq N$. 
Moreover, any morphism of $R$-algebras $f\colon A \to B$ induces a morphism of semirings
\begin{equation}\label{eq: induced morphism between semirings of fg subsemimodules}
    M_R(f)\colon M_R(A)\longrightarrow M_R(B)
\end{equation}
defined by mapping the finitely generated $R$-subsemimodule $\langle a_1,\ldots,a_n\rangle_R$ of $A$ to~$\langle f(a_1),\ldots,f(a_n)\rangle_R$. 

Let us now consider the map
\begin{equation}\label{eq: definition of universal valuation}
    v_R\colon A \longrightarrow M_R(A),\qquad a\longmapsto\langle a\rangle_R.
\end{equation}
It is related to the previous subsection by the following observation.

\begin{lemma}\label{lem: universal G-valuation is G-valuation} 
The map $v_R$ is a $G$-valuation on $A$ with~$\iota(R)\subseteq A_{v_R}$.
\end{lemma}
\begin{proof}
    We have $v_R(0_A)=\{0_A\}$, which is the zero element of~$M_R(A)$, and $v_R(1_A)=\iota(R)$, which is the multiplicative unit element of~$M_R(A)$.
    For all $a,b\in A$ we have
    \[
    v_R(a)v_R(b)=\langle a\rangle_R\cdot\langle b\rangle_R=\langle ab\rangle_R=v_R(ab).
    \]
    Further,
    \[
    v_R(a+b)=\langle a+b\rangle_R \subseteq\langle a\rangle_R+\langle b\rangle_R = v_R(a) + v_R(b).
    \]
    Since the partial order on $M_R(A)$ is precisely~$\subseteq$, we conclude that $v_R$ satisfies all axioms in \cref{def: G-valuation}.
    Finally, for all $r\in R$ we have that $v_R(\iota(r))=\langle\iota(r)\rangle_R\subseteq\iota(R)$, so that $v_R(\iota(r))\preceq 1_{M_R(A)}$, as wanted.
\end{proof}

We call the map $v_R$ in \eqref{eq: definition of universal valuation} the \emph{universal $G$-valuation} of the $R$-semi\-algebra~$A$.
The terminology is justified by the following result, which was given by MacPherson in the case of rings in \cite[Example~3.16]{Macpherson2020} and generalized to semirings by the last author in \cite{Xarles}.
In the specific case~$R=\mathbb{Z}$, it gives an explicit description of the universal object considered in \cite[Proposition~2.5.4]{Giansiracusa2X_2016}. 

\begin{theorem}\label{thm: universal G-valuation satisfies a universal property}
    For every $G$-valuation $v\colon A\rightarrow S$ such that~$\iota(R)\subseteq A_{v}$, there exists a unique morphism of semirings $f\colon M_R(A) \rightarrow S$ such that~$v=f\circ v_R$.
\end{theorem}
\begin{proof}
    We define $f$ by taking, for each $M=\langle a_1,\dots,a_n\rangle_R\in M_R(A)$,
    \[
    f(M)=\sum_{i=1}^n v(a_i).
    \]
    We first check that $f$ is well-defined.
    For this, suppose $M=\langle a_1,\dots,a_n\rangle_R=\langle b_1,\dots,b_m\rangle_R$.
    Then, each $a_i$ can be written as a finite sum $a_i=\sum_{j=1}^m r_{ij}b_j$ with~$r_{ij}\in R$.
    Hence
    \begin{equation}\label{eq: sum in universal G-valuation satisfies a universal property}
    v(a_i)=v\Bigl(\sum_j r_{ij} b_j\Bigr)
    \preceq \sum_j v(r_{ij}b_j)
    \preceq \sum_j v(\iota(r_{ij}))\,v(b_j)
    \preceq \sum_j v(b_j),
    \end{equation}
    since $\iota(R)\subseteq A_v$ implies that~$v(\iota(r_{ij}))\preceq1_S$, and the partial order on $S$ respects sums and products.
    For the same reason, summing \eqref{eq: sum in universal G-valuation satisfies a universal property} over $i=1,\ldots,n$ yields $\sum_i v(a_i)\preceq \sum_j v(b_j)$.
    The same argument with the roles interchanged gives the reverse relation, and so the desired equality.
    
    It is easy to see that $f$ is a semiring morphism and by definition we have
    \[
    (f\circ v_R)(a)=f(\langle a\rangle_R)=v(a)
    \]
    for all~$a \in A$, proving that~$v=f\circ v_R$.
    Uniqueness is obvious from this last requirement and from the fact that each $M=\langle a_1,\dots,a_n\rangle_R\in M_R(A)$ can be written as~$M=\langle a_1\rangle_R+\ldots+\langle a_n\rangle_R$.
\end{proof}

\begin{example}
    When $R=\mathbb{N}$ we get a universal object for semirings, when $R = \mathbb{B}$ a universal object for idempotent semirings, and when $R=A$ a universal object for finitely generated ideals of $A$ with minimal $1$-bounded valuation.
\end{example}

The following key result topologically identifies the spectrum of $A$ and the $k$-spectrum of the associated idempotent semiring~$M_R(A)$ through the use of the universal $G$-valuation of~$A$.
In the case of rings and ideals, this relation was established in \cite[Proposition~3.10]{BarilBoudreau_Garay}. 

\begin{theorem}\label{thm: G-valuations induce homeomorphism sp spec}
For any $R$-semialgebra $A$ the morphism
\[
v_R^*\colon\Sp M_R(A)\longrightarrow \Spec A
\]
is a homeomorphism. 
\end{theorem}

In the proof of \cref{thm: G-valuations induce homeomorphism sp spec}, we will use the observation that the topology of $\Sp M_R(A)$ is generated by a much smaller class than the one of all principal open sets.
Namely, it is sufficient to take those defined by $R$-subsemimodules of~$A$ generated by a single element.
The precise result is as follows.

\begin{lemma}\label{lem: D for finitely generated submodules}
For all~$a_1,\ldots,a_n\in A$, we have
\[
\bD(\langle a_1,\ldots,a_n\rangle_R)=\bigcup_{i=1}^n\bD(\langle a_i\rangle_R).
\]
In particular, the family $\big\{\bD(\langle a\rangle_R)\mid a\in A\big\}$ is a basis for the Zariski topology of~$\Sp M_R(A)$.
\end{lemma}
\begin{proof}
    The second claim is immediate from \cref{lem: principal open sets are basis for Zariski topology of Sp} and from the first one.
    To prove the latter, note that a prime $k$-ideal $\mathfrak{p}$ of $M_R(A)$ satisfies $\mathfrak{p}\in \bD(\langle a_1,\ldots,a_n\rangle_R)$ if and only if the associated function $\chi_\mathfrak{p}\colon M_R(A)\rightarrow \mathbb{B}$ from \cref{rem: G-valuation and Booleans} is such that~$\chi_\mathfrak{p}(\langle a_1,\ldots,a_n\rangle_R)=1$.
    Since $\chi_\mathfrak{p}$ is a homomorphism of semirings because of the proof of \cref{lem: Sp of idempotent semiring} and $\langle a_1,\ldots,a_n\rangle_R=\langle a_1 \rangle_R+\ldots+ \langle a_n \rangle_R$, we deduce that the previous equality is equivalent to
    \[
    1=
    \chi_\mathfrak{p}(\langle a_1,\ldots,a_n\rangle_R)
    =
    \sum_{i=1}^n \chi_\mathfrak{p}(\langle a_i\rangle_R).
    \]
    This condition is equivalent to saying that there exists some $i=1,\ldots,n$ such that $\chi_\mathfrak{p}(\langle a_i\rangle_R)=1$.
    This happens precisely when~$\langle a_i\rangle_R\notin \mathfrak{p}$, that is when $\mathfrak{p}\in\bD(\langle a_i\rangle_R)$. 
\end{proof}

\begin{proof}[Proof of~\cref{thm: G-valuations induce homeomorphism sp spec}]
    Applying \cref{prop: valuation induces morphism of spectra} and \cref{lem: universal G-valuation is G-valuation}, we already know that the map
    \[
    v_R^*\colon \Sp M_R(A)\longrightarrow \Spec A,
    \qquad 
    \mathfrak{p}\longmapsto v_R^{-1}(\mathfrak{p})
    \]
    is well-defined and continuous.
    To prove it is bijective, define for each $\mathfrak{q}\in\Spec A$ the set
    \[
    \mathfrak{p}_{\mathfrak q}=\{M\in M_R(A)\mid M\subseteq \mathfrak q\}.
    \]
    It is easily seen that $\mathfrak{p}_\mathfrak{q}$ is a subtractive prime ideal in $M_R(A)$ that satisfies $v_R^*(\mathfrak p_{\mathfrak q})=\mathfrak q$.
    Conversely, take~$\mathfrak{p}\in\Sp M_R(A)$.
    For all $M=\langle a_1,\ldots,a_n\rangle_R\in M_R(A)$ it holds that
    \[
    \mathfrak{p}\notin\bD(M)
    \iff
    \mathfrak{p}\notin\bD(\langle a_i\rangle_R)\ \forall i=1,\ldots,n
    \]
    thanks to \cref{lem: D for finitely generated submodules}.
    Thus $M\in\mathfrak{p}$ if and only if $\langle a_i\rangle_R\in\mathfrak{p}$ for all $i=1,\ldots,n$.
    But the last condition is equivalent to the fact that $\langle a\rangle_R\in\mathfrak{p}$ for all $a\in M$ since
    \[
    \langle a\rangle_R+\langle a_1\rangle_R+\ldots+\langle a_n\rangle_R=\langle a_1\rangle_R+\ldots+\langle a_n\rangle_R
    \]
    and $\mathfrak{p}$ is substractive.
    We deduce that $M\in\mathfrak{p}$ if and only if $v_R(a)\in\mathfrak{p}$ for all $a\in M$, and so
    \[
    \mathfrak{p}_{v_R^*(\mathfrak{p})}
    =
    \{M\in M_R(A)\mid v_R(M)\subseteq\mathfrak{p}\}
    =
    \mathfrak{p},
    \]
    concluding the proof of bijectivity.

    Finally, we need to check that $v_R^*$ is an open map.
    To do this, notice that for all $a\in A$ we have
    \[
    \mathfrak{p}\in \bD(\langle a\rangle_R)
    \iff
    v_R(a)\notin\mathfrak{p}
    \iff
    a\notin v_R^{-1}(\mathfrak{p})
    \iff
    v_R^*(\mathfrak{p})\in D(a).
    \]
    Therefore,
    \begin{equation}\label{eq: image of principal open by universal valuation}
    v_R^*(\bD(\langle a\rangle_R))
    =
    D(a),
    \end{equation}
    which is an open set in~$\Spec A$.
    By \cref{lem: D for finitely generated submodules}, this is enough to conclude that $v_R^*$ is an open map, finishing the proof.
\end{proof}

\begin{remark}
    Another way to look at the bijection in \cref{thm: G-valuations induce homeomorphism sp spec} is to recall that, with the same identifications as in the proof of \cref{prop: valuation induces morphism of spectra}, the map $v_R^*$ corresponds to
    \[
    \mathrm{Hom} (M_R(A),\mathbb{B}) \longrightarrow \mathrm{Val} (A,\mathbb{B}),\qquad f\longmapsto f\circ v_R.
    \]
    Since any $G$-valuation with values in $\mathbb{B}$ has integral image, the fact that such assignment is a bijection follows from the universal property in \cref{thm: universal G-valuation satisfies a universal property}.
\end{remark}

\subsection{A presheaf associated to the universal valuation}

Let us consider again a semiring $R$ and an $R$-semialgebra~$A$.
The homeomorphism in \cref{thm: G-valuations induce homeomorphism sp spec} suggests the definition of a distinguished presheaf $\mathcal{F}$ on~$\Sp M_R(A)$.
Namely, for any open set~$U\subseteq\Sp M_R(A)$, we can consider the corresponding open set $v_R^*(U)\subseteq\Spec A$ and define
\[
\mathcal{F}(U)=M_R(\mathcal{O}_{\Spec A}(v_R^*(U))).
\]
This is meaningful, as $\mathcal{O}_{\Spec A}(v_R^*(U))$ is naturally an $A$-semialgebra, and hence an~$R$-semialgebra.

\begin{remark}
    Notice that, even in the case~$A=R$, we are not considering the presheaf assigning to every open $U$ the idempotent semiring of finitely generated ideals in~$\mathcal{O}_{\Spec A}(v_R^*(U))$, as it is done in \cite[Section~4.2]{BarilBoudreau_Garay}, but the larger idempotent semiring of its finitely generated $A$-subsemimodules. 
\end{remark}

The presheaf $\mathcal{F}$ rejoices pleasant properties on the basis of the topology of $\Sp M_R(A)$ given in \cref{lem: D for finitely generated submodules}.
Indeed, because of \eqref{eq: image of principal open by universal valuation} and \cref{thm: equivalence of definitions} we have for all~$a\in A$ that
\[
\mathcal{F}(\bD(\langle a\rangle_R))\cong M_R(A[a^{-1}]).
\]
In particular, its global sections are, as desirable,
\[
\mathcal{F}(\Sp M_R(A))\cong M_R(A).
\]
Moreover, the semirings of local sections of $\mathcal{F}$ on these principal open sets are localizations of the semiring of global sections, thanks to the following.

\begin{lemma}
    For all $a\in A$ there is an isomorphism
    \[
    M_R(A[a^{-1}])\cong M_R(A)\big[v_R(a)^{-1}\big].
    \] 
\end{lemma}
\begin{proof}
    Recall that the localization $A[a^{-1}]$ comes equipped with a natural homomorphism~$\varphi_a\colon A\to A[a^{-1}]$.
    As in \eqref{eq: induced morphism between semirings of fg subsemimodules} this induces a semiring homomorphism $M_R(\varphi_a)\colon M_R(A)\to M_R(A[a^{-1}])$ mapping $v_R(a)$ to
    \[
    M_R(\varphi_a)(v_R(a))
    =
    M_R(\varphi_a)(\langle a\rangle_R)
    =
    \langle \varphi_a(a)\rangle_R.
    \]
    This element is invertible in~$M_R(A[a^{-1}])$, with inverse $\langle 1_A/a\rangle_R$.
    Thus, the universal property of localization gives a unique morphism
    \[
    M_R(A)\big[v_R(a)^{-1}\big]\longrightarrow M_R(A[a^{-1}])
    \]
    recovering $M_R(\varphi_a)$ from the localization morphism.

    Conversely, it is easy to check with the help of \cref{lem: universal G-valuation is G-valuation} that the map $v\colon A[a^{-1}] \to M_R(A)\big[v(a)^{-1}\big]$ defined by $v(b/a^n)=v_R(b)/v_R(a)^n$ for all $b\in A$ is a $G$-valuation for which all elements of $R$ are integral.
    Therefore, thanks to \cref{thm: universal G-valuation satisfies a universal property} there exists a unique semiring homomorphism
    \[
    M_R(A[a^{-1}])\longrightarrow M_R(A)\big[v_R(a)^{-1}\big]
    \]
    recovering $v$ from the universal $G$-valuation of~$A[a^{-1}]$.

 Since both maps come from universal properties, one can see that their composition is the identity, concluding the proof.
 \end{proof}

However, given $a_i \in A$ for $i\in I$ such that $\bigcup_{i\in I} \bD(\langle a_i\rangle_R) = \Sp M_R(A)$, the morphism
\begin{equation}\label{eq: sequence with equalizer for universal G-valuation}
    M_R(A) \longrightarrow\operatorname{Eq}\bigg(\prod_{i\in I} M_R(A[a_i^{-1}]) \rightrightarrows \prod _{i,j\in I}  M_R( A[(a_ia_j)^{-1}])\bigg)
\end{equation}
does not need to be an isomorphism.
This means that the presheaf $\mathcal{F}$ is not in general a sheaf on the basis of open sets described in \cref{lem: D for finitely generated submodules}.


\begin{remark} 
This approach ought to be equivalent to the one presented by \cite{Giansiracusa_this_volume} and reinterprets bend relations.
The problem of the morphism in \eqref{eq: sequence with equalizer for universal G-valuation} not being an isomorphism is the same as the globalization one in \cite{Giansiracusa_this_volume}, seen from a different perspective.
\end{remark}

Finally, let us note that the presheaf $\mathcal{F}$ defined here is not in general equal to the presheaf $\mathcal{L}_{M_R(A)}$ defined in \cref{sec: structure sheaf for Sp}.
In fact, their sections can differ on principal open sets, since it can happen that
\[
M_R(A)\big[v_R(a)^{-1}\big]\not\cong \widetilde{S}_{v_R(a)}^{-1}M_R(A)
\]
for some~$a\in A$, as the next example shows.

\begin{example}
Consider the ring $R=K[x,y]$ of polynomials in two variables over a field~$K$, and let~$A=R$.
Then, $M_R(R)$ is the semiring of (finitely generated) ideals of~$R$.
Consider also the polynomial $x$ in~$R$ and the maximal ideal~$\mathfrak{m}=\langle x,y\rangle_R=xR+yR$.

On the one hand, by \cref{lem: D for finitely generated submodules}
\[
\bD(\mathfrak{m})=\bD(\langle x\rangle_R)\cup \bD(\langle y\rangle_R) \supseteq \bD(\langle x\rangle_R)=\bD(v_R(x)).
\]
Therefore, seeing it as an element of~$M_R(R)$, we have $\mathfrak{m}\in \widetilde{S}_{v_R(x)}$ and so its image is invertible in the localization~$\widetilde{S}_{v_R(x)}^{-1}M_R(R)$. 

On the other hand, since the multiplicative submonoid of $M_R(R)$ generated by $v_R(x)$ is saturated, the only invertibles in $M_R(R)\big[v_R(x)^{-1}\big]$ are the images of powers of~$v_R(x)$, because of \cref{lem: Saturation}.
In particular, the image of $\mathfrak{m}$ is not invertible in~$M_R(R)\big[v_R(x)^{-1}\big]$.
\end{example}

\subsection{Globalizing $G$-valuations}\label{subs: global G-valuation}

The constructions from the previous subsections can be generalized to the case of the semischemes defined in \cref{sec: semiring schemes}.

\begin{definition}
    A \emph{$G$-valuation} on a semischeme $(\mathfrak{X},\mathcal{O}_\mathfrak{X})$ is the datum of an idempotently semiringed space $(\mathfrak{S}
    ,\mathcal{V}_\mathfrak{S})$ (meaning, a topological space equipped with a sheaf of idempotent semirings),  
    and of a morphism $v\colon\mathfrak{S}\rightarrow \mathfrak{X}$ such that
    \[
    v^\#(U)\colon\mathcal{O}_{\mathfrak{X}}(U)\to\mathcal{O}_{\mathfrak{S}}(v^{-1}(U))
    \]
    is a $G$-valuation for every open set~$U\subseteq \mathfrak{X}$.
    
    We also say that the $G$-valuation is \emph{$1$-bounded} if for every open set $U\subseteq\mathfrak{X}$ the above map is a 1-bounded $G$-valuation.
\end{definition}

With this definition, given any morphism of semischemes $\mathfrak{X}\rightarrow\mathfrak{Y}$ one can construct a \emph{relative universal $G$-valuation} $v_{\mathfrak{Y}}$ fitting in the diagram

\begin{center}
    \begin{tikzcd}[column sep=small]
    M_{\mathfrak{Y}}(\mathfrak{X}) \arrow{r}{v_{\mathfrak{Y}}}  \arrow{rd}{} 
    & \mathfrak{X} \arrow{d}{} \\
    & \mathfrak{Y}
\end{tikzcd}
\end{center}
with the diagonal arrow being a $1$-bounded $G$-valuation, and satisfying a universal property analogous to~\cref{thm: universal G-valuation satisfies a universal property}.

We refer the interested reader to a forthcoming paper \cite{Xarles} by the last author for the details.

\begin{small}
 \bibliographystyle{alpha}
 \bibliography{ref}
\end{small}

\end{document}